%% file: Main.tex
\documentclass[reqno]{amsart}
\pdfoutput=1 
\usepackage{preprintStyle}

\usepackage{bm}
\usepackage{comment}

\usepackage{algorithmic}

\input{def}	

\usepackage{graphicx}
\usepackage{color}
\usepackage{circuitikz}
\usepackage{tikz}
\usetikzlibrary{calc,positioning,shapes}
\usetikzlibrary{patterns,decorations.pathmorphing,decorations.markings}
\usetikzlibrary{external}
\pgfplotsset{compat=newest}
\usepackage[margin=10pt,font=small,labelfont=bf,labelsep=endash]{caption}
\usepackage{subcaption}

\usetikzlibrary{intersections, backgrounds}
\usetikzlibrary{circuits}
\usetikzlibrary{circuits.ee.IEC}
\usepackage{pgfplots}
\usepackage{siunitx}
\pgfplotsset{compat=newest}

\definecolor{color0}{rgb}{0.12156862745098,0.466666666666667,0.705882352941177}
\definecolor{color1}{rgb}{1,0.498039215686275,0.0549019607843137}
\definecolor{color2}{rgb}{0.172549019607843,0.627450980392157,0.172549019607843}
\definecolor{color3}{rgb}{0.83921568627451,0.152941176470588,0.156862745098039}
\definecolor{color4}{rgb}{0.580392156862745,0.403921568627451,0.741176470588235}
\definecolor{color5}{rgb}{0,0,0}

\definecolor{mycolor1}{rgb}{0.00000,0.44700,0.74100}
\definecolor{mycolor2}{rgb}{0.85000,0.32500,0.09800}
\definecolor{mycolor3}{rgb}{0.92900,0.69400,0.12500}
\definecolor{mycolor4}{rgb}{0.46600,0.67400,0.18800}
\definecolor{mycolor5}{rgb}{0.49400,0.18400,0.55600}

\newcommand{\lineWidth}{1.2pt}
\newcommand{\imageWidth}{2.0in}
\newcommand{\imageHeight}{1.8in}

\title[\CIM and \MOR for parametric linear control systems]{Contour integral methods and model order reduction for parametric linear control systems}
\author{Serkan Gugercin${}^\dag$}
\author{Mattia Manucci${}^\star$}

\address{${}^{\dag}$ Department of Mathematics and Division of Computational Modeling and Data Analytics, Academy of Data
Science, Virginia Tech, Blacksburg, VA, USA}
\email{gugercin@vt.edu}
\address{${}^{\star}$ Institute for Applied and Numerical Mathematics, Karlsruhe Institute of Technology, 76131 Karlsruhe, Germany}
\email{mattia.manucci@kit.edu}

\date{\today}

\begin{document}

\begin{abstract}
This paper introduces a contour integral method (CIM) for efficiently computing outputs of parametric linear systems in control form over specified time intervals and to a user-prescribed accuracy. The CIM approximates the inverse Laplace transform via a quadrature rule applied along a modified integration contour. For parametric systems, we show how CIM integrates effectively with projection-based model order reduction (MOR) where a greedy algorithm builds the projection spaces following an error estimate we derive for this setting. We additionally demonstrate that the developed projection framework naturally enforces Hermite interpolation conditions. This combination substantially lowers the cost of evaluating the input-output relations across the parameter domain, for a wide range of input functions, and for initial conditions well captured by a low-dimensional subspace.. We demonstrate the accuracy and efficiency of the approach on benchmark non-parametric and parametric control systems, comparing against state-of-the-art projection-based MOR methods.
\end{abstract}

\maketitle
{\footnotesize \textsc{Keywords:} parametric linear control systems, contour integral methods, projection-based model order reduction, error estimates, limited time interval} 

{\footnotesize \textsc{AMS subject classification:} 37M05, 65E10, 65F60, 65G20, 65L70, 93A15 }



\section{Introduction}
Dynamical systems serve as a fundamental structure for modeling and controlling a wide range of complex systems of scientific interest and industrial value. These include domains such as heat transfer, fluid dynamics, chemical reaction flows, biological systems, signal transmission and interference in electronic circuits, wave dynamics, vibration management in large structures, and option pricing. Numerical simulations of related models are among the limited tools available for understanding intricate physical processes. However, the demand for increased accuracy necessitates adding more detail in the modeling phase, leading to more extensive and intricate dynamical systems. The persistent pursuit of optimizing system performance requires the simulation of numerous possible system configurations. In such large-scale contexts, the execution of numerous simulations can impose significant computational demands. This leads to the necessity of developing \emph{reduced order models} (\ROMs), i.e. efficient representations that are quick to evaluate while accurately approximating high-resolution simulations, often referred to as \emph{full order models} (\FOMs). 
This need has driven \emph{model order reduction} (\MOR) that comprises the wide range of mathematical techniques developed to construct such simplified models.

In this paper, we focus on the broad class of parametric linear dynamical systems in control form (also known as the state-space form), i.e., systems of the form  
\begin{equation}
	\label{eqn:LTI:cf:par}
     \left\{\quad \begin{aligned}
		\fE(\prmtr)\dot{\stx}(t,\prmtr)\;  &= \;\fA(\prmtr) \stx(t,\prmtr) +\fB(\prmtr)\inp(t), & \stx(0) &= \stx_0, \\
		\out(t,\prmtr) \;&= \;\fC(\prmtr)\stx(t,\prmtr)+\fD(\prmtr)\inp(t),\\
	\end{aligned}\right.
\end{equation}
where    $\prmtr \in \prmtrSet \subseteq \R^{\prmtrDim}$ is the so-called parameter vector, $\stx(t,\prmtr)\in\R^{\stateDim}$, $\inp(t)\in\R^{\inpDim}$, and $\out(t,\prmtr)\in\R^{\outDim}$ are, respectively, the state, input, and output vectors while $\fE(\prmtr)\in\R^{\stateDim\times\stateDim}$ is invertible for all $\prmtr\in\prmtrSet$, $\fA(\prmtr)\in\R^{\stateDim\times\stateDim}$, $\fB(\prmtr)\in\R^{\stateDim\times\inpDim}$, $\fC(\prmtr)\in\R^{\outDim\times\stateDim}$, and $\fD(\prmtr)\in\R^{\outDim\times\inpDim}$ for all $\prmtr\in\prmtrSet$. The parameters
may enter the models in many ways, representing, for example, material properties,
system geometry, system configuration, initial conditions, and boundary conditions. The goal is to rapidly and accurately characterize the response (the output) of the system for different values of the parameters, of the input function, and the initial solution for all times in a prescribed interval of interest $[T,\Lambda T]$, $T\in\R_{>0}$ and $\Lambda\ge1$. To accomplish this goal, one aims to replace \eqref{eqn:LTI:cf:par} with a suitable \ROM of the form
\begin{equation}
	\label{eqn:LTI:cf:par:red}
     \left\{\quad \begin{aligned}
		\redfE(\prmtr)\dot{\stx}_{\mathrm{r}}(t,\prmtr)\;  &= \;\redfA(\prmtr) \redstx(t,\prmtr) +\redfB(\prmtr)\inp(t), & \redstx(0) &= \stx_{\mathrm{r},0}, \\
		\redout(t,\prmtr) \;&= \;\redfC(\prmtr)\redstx(t,\prmtr)+\fD(\prmtr)\inp(t),\\
	\end{aligned}\right.
\end{equation}
where now $\redstx(t,\prmtr)\in\R^{\stateDimRed}$ is the reduced state with $\stateDimRed\ll\stateDim$, $\redout(t,\prmtr)$ is the approximate output, and $\redfE(\prmtr)\in\R^{\stateDimRed\times\stateDimRed}$, $\redfA(\prmtr)\in\R^{\stateDimRed\times\stateDimRed}$, $\redfB(\prmtr)\in\R^{\stateDimRed\times\inpDim}$, $\redfC(\prmtr)\in\R^{\outDim\times\stateDimRed}$ are the reduced parametric matrices, such that $\redout(t,\prmtr)$ approximates, under prescribed precision $\tol$ and for a correspondingly defined error measure, $\out(t,\prmtr)$ for a wide range of input $\inp$, for all $\prmtr \in \prmtrSet $, and for all $t\in[T,\Lambda T]$.

\subsection{Literature review}  Among reduced-order modeling techniques, projection-based methods stand out as a prominent class. These can be organized into two main frameworks: the \emph{reduced basis method} (\RBM), originally developed to address the computational reduction of parametric \PDEs, and \emph{system-theoretic} methods, which target the reduction of dynamical systems by preserving their input-output behavior.

In both cases, the strategy follows an offline/online decomposition. The \emph{offline} phase consists of expensive computations whose cost scales with the size of the \FOM. The goal is to build one or more reduced subspaces onto which the problem is subsequently projected. These spaces are usually of a dimension much smaller than
the space associated with the \FOM. After the reduced spaces have been generated, the quantity of interest corresponding to any new parameter instance $\prmtr$, input function $\inp$, and/or initial condition $\stx_0$ can be efficiently evaluated during the \emph{online} phase. This is accomplished by solving a \ROM obtained through a Galerkin or Petrov--Galerkin projection of the \FOM onto the reduced spaces. Regarding the \RB method, we refer the reader to, e.g., \cite{HesRS16,Haa17} and the references therein. For the specific case of parametric time-dependent problems, treated from a reduced-basis perspective, we mention the survey \cite{GlaMU17}, where two methodologies are described and compared. The first one is based on a time-stepping solver
like a Runge-Kutta method in the offline phase. The reduced basis is then usually
formed via a greedy algorithm, that is, an iterative procedure where at each iteration one new basis function is added so as to improve the overall approximation capability of the basis set. The greedy strategy is then combined with
compression of the numerical solution computed along the grid through a \emph{proper
	orthogonal decomposition} (\POD), which is a singular value decomposition of the array
containing the solution vectors \cite{DroHO12,Haa13}. The drawback of these time-stepping schemes
is that, in order to approximate the solution at a certain time $T = t_n$, one needs to
compute an approximation of the solution, for both the full and the reduced problem,
at grid points $0 < t_1 < t_2 < \cdot \cdot \cdot  < t_n$, which would be particularly demanding if $T$ is
large and/or there are many grid points. Moreover, even if the \FOM has some stability properties (for instance, asymptotic stability), it is not guaranteed, in general, that the resulting \ROM retains these features; see \cite{Emb19}. The second approach consists of treating time as an additional variable, which
results in a problem of dimension $d + 1$, where $d$ denotes the spatial dimension. The reduced basis is formed by a standard greedy algorithm and then the full problem is
projected (in the sense of Galerkin or Petrov--Galerkin) onto the reduced space \cite{UrbP12}. It
is well known that the size of the discrete problem grows exponentially with respect
to the number of variables when keeping the same accuracy (the so-called curse
of dimensionality); therefore, the computational cost for this approach can become
prohibitive in the offline phase. 

A survey of \MOR\ techniques for parametric dynamical systems of the form \eqref{eqn:LTI:cf:par} is given in \cite{BenGW15}, which draws together contributions from different communities to survey the state of the art in parametric model reduction; among these, it offers a particularly thorough treatment of system-theoretic methods. In this context, the review discusses, among other topics, how to construct system-theoretic reduced spaces that account for parameter variations, and for certain special cases, even characterizes optimal choices of such spaces. Recently, advances have been made toward broadening this picture, with several approaches proposing optimal or greedy-type selection strategies to construct reduced spaces that account for parameter dependence more generally, see, e.g., \cite{GosGU21,MliG23,FenCB24,BelCN25,FenAB17,HunMMS22,MliBG24} and the references therein. 
Within the system theory community, interpolatory methods based on sampling in the frequency domain have long been among the most widely used approaches for constructing reduced-order models; see \cite{AntBG20} for a comprehensive treatment. Recently, this frequency-domain perspective has also been adopted within the reduced basis community: For example, \cite{GugM23} employs frequency-domain sampling within a contour integral method (\CIM) to approximate the inverse Laplace transform, enabling efficient approximation of the state at a specific time instant for linear parabolic \PDEs. This connection between the frequency-domain sampling and contour integral methods for evaluating outputs at prescribed times is a central starting point for our work here. The system-theoretic community has also addressed the construction of \ROMs that are accurate specifically for $t$ in a given time interval $[0,T]$ or $[T,\Lambda T]$. For example, for a discussion of \emph{time-limited balanced truncation} (\TLBT), we refer to \cite{Kur18} and the references therein; for time-limited $\mathcal{H}_2$-optimal \MOR, see \cite{SinG18,GoyR19}. However, to our knowledge, the extension of these time-limited frameworks to the parametric setting has not yet been explored.

As already mentioned, although \MOR can strongly reduce the size of dynamical systems, their accurate time integration may still require an expensive computational effort for accurate real-time predictions due to small time-step constraints or long time horizons. The efficiency limit imposed by standard time integration is also discussed in \cite{HaaKOSW23} and motivates the authors to replace the time discretization scheme with a neural network. Moreover, being efficient, without compromising accuracy, in the time integration, is particularly relevant in those applications where one would be interested in computing the solution only for a given time or in prescribed time windows. Employing standard time-step integrators results in evaluating the state variable in many intermediate time steps that are not of interest. A prominent example application is finance models where the evolution in time of the option price is of interest only in precise time instances where the option can be sold or bought. This motivated the work \cite{GugM23} where a certified projection-based \MOR based on the approximation of the inverse Laplace transform is proposed to deal with the parametric source problem arising from the linear parabolic type \PDEs.

\subsection{Main contributions} This paper introduces an innovative framework designed for the efficient evaluation of outputs from parametric \LTI control systems, with a focus on error certification, at a specific time $T$ or within a time window $[T,\Lambda T]$, where $\Lambda>1$. This goal is achieved by using \emph{contour integral methods} (\CIMs) for time integration. These methods facilitate a direct parallelization of the primary computational tasks \cite{Guglielmi2020,Guglielmi2021} and are inherently well suited for deployment with a non-linear \MOR framework. Initially, we address the non-parametric scenario by introducing the contour integral methods for linear control systems (\CIM-\LCS) approach. We demonstrate that, given fixed initial conditions that belong to a subspace of relativly small dimension, the computational cost of assessing the input-output mapping is solely based on the input, output, and space of the initial solution dimensionalities. This method inherently accommodates varying non-zero initial conditions in a subsapce. Subsequently, we incorporate the parameter dependency and effectively integrate the approach for the non-parametric scenario with a validated projection-based non-linear \MOR method leveraging a greedy algorithm, resulting in the \CIM-\MOR for parametric linear control systems. Our proposed algorithm is based on the development of a new error estimate, introduced in \Cref{lemma:err:est}, which we demonstrate to be effectively assessable. This ensures that the reduction error, as defined appropriately, will remain below a user-specified accuracy threshold. Ultimately, this enables the reduction of the state's dimensionality, facilitating efficient computation of the input-output relationship for each parameter within the parametric domain. This reduction in dimensionality also holds for initial conditions that can be effectively represented within a low-dimensional subspace and a general class of input functions. We also show that, for the proposed projection spaces, certain types of Hermite interpolation conditions hold; see \Cref{lem:int:cond}.

\subsection{Outline of the paper} In \Cref{sec2}, we first review the fundamental tools of \CIM and subsequently explore their extension to the \LCS class to achieve efficient and highly precise approximations. \Cref{sec3} addresses the parametric scenario, where we begin by outlining the general assumptions of our framework and subsequently introduce the \CIM-\MOR algorithm, following the derivation of the error estimate in \Cref{lemma:err:est}. We analyze the computational complexity related to the new error bound in \Cref{sec:3.3} and then present numerical findings for various benchmark cases in \Cref{sec4}. Finally, our conclusions are elaborated in \Cref{sec5}.

\subsection{Notation} $\fI_{\stateDim}$ denotes the identity matrix of size $\stateDim$. For a generic vector $\fv$, we denote by $\|\fv\|$ the standard Euclidean norm of $\fv$. The same holds for matrices, i.e., for a generic matrix $\fA$ we have $\|\fA\|$ being the induced Euclidean norm.

\section{\CIM for \LTI control systems}\label{sec2}
Let us rewrite \eqref{eqn:LTI:cf:par} without parameter dependence, i.e.,
\begin{equation}
	\label{eqn:LTI:cf}
     \left\{\quad \begin{aligned}
		\fE\dot{\stx}(t)\;  &= \;\fA \stx(t) +\fB\inp(t), & \stx(t_0) &= \stx_0, \\
		\out(t) \;&= \;\fC\stx(t)+\fD\inp(t),\\
	\end{aligned}\right.
\end{equation}
We assume that $\fE^{-1}\fA$ is Hurwitz (i.e., all eigenvalues have negative real part) and that the initial condition is an element of a subspace $\calF$ of dimension $\stateDimRed_{\calF}\ll\stateDim$. Therefore, $\stx_0=\inMat\inSolSub$, where $\inMat\in\R^{\stateDim\times\stateDimRed_{\calF}}$ is a matrix with orthonormal columns that span the whole $\calF$ and $\inSolSub\in\R^{\stateDimRed_{\calF}}$.
\begin{remark}
   We emphasize that the assumption that $\fE^{-1}\fA$ is Hurwitz is not required for our methodology to be applicable. In fact, our methodology, which is formulated for finite time windows, is applicable to systems whose associated matrices possess eigenvalues with positive real parts. Nevertheless, working with unstable systems continues to impose constraints on the admissible choice of the time horizon \(T\) for which the numerical approximation remains unperturbed by the exponential growth induced by unstable eigenvalues acting on finite-precision round-off errors.
\end{remark}
\noindent Recall that we are interested in the case where $\stateDim$ is large, say $\stateDim\gg 10^3$, and $\inpDim,\outDim\ll\stateDim$. Our aim is, for different control functions $\inp$ and for all the initial conditions from the subspace $\calF$, to efficiently evaluate the output $\out$ for all $t$ in $[T,\Lambda T]$, $\Lambda\ge1$, without evaluating it for any time $t<T$. It is evident that in pursuit of this objective, the matrix $\fD$, commonly known as the feedthrough matrix, is irrelevant and its influence does not necessitate any approximation. We assume that the Laplace transform of $\inp$ exists and that it allows for a bounded analytic continuation into an appropriate area of the complex plane, excluding the generalized eigenvalues of the matrix pair $(\fA,\fE)$. We subsequently apply the Laplace transform operator $\lapOp$ to \eqref{eqn:LTI:cf}. Denoting the Laplace transforms of $\stx$, $\out$, and $\inp$ by $\hat{\stx}\vcentcolon=\lapOp(\stx)$, $\hat{\out}\vcentcolon=\lapOp(\out)$, and $\hat{\inp}\vcentcolon=\lapOp(\inp)$, respectively, this operation yields
\begin{equation}\label{eqn:LT:cf}
	 \left\{\quad \begin{aligned}
	\hat{\stx}(\lapVar)\;&=\;\left(\lapVar\fE-\fA\right)^{-1}\left(\fE\inMat\inSolSub+\fB\hat{\inp}(\lapVar)\right),\\
	\hat{\out}(\lapVar)\;&=\;\fC\hat{\stx}(\lapVar)+\fD\hat{\inp}(\lapVar).
		\end{aligned}\right.
\end{equation}
To retrieve the solution in terms of the time variable, the inverse Laplace transform is employed, resulting in
\begin{equation}\label{eqn:inv:Lap}
	\out(t)\;=\;\frac{1}{2\pi\imagunit}\int_{\gamma-\imagunit\infty}^{\gamma-\imagunit\infty} \ee^{\lapVar t}\fC\hat{\stx}(\lapVar)\;\domega\;+\;\fD\inp(t),
\end{equation}
for certain $\gamma>0$ where $\imagunit$ denotes the imaginary unit. Then, assuming
\begin{enumerate}[label=\roman*)]
    \item the singularity of the integrand function in \eqref{eqn:inv:Lap} lies in a sectorial region of the complex plane with $\real(\lapVar)< \gamma$ and
    \item the integrand function in \eqref{eqn:inv:Lap} decays as $|\lapVar|\rightarrow\infty$,
\end{enumerate}
following \cite{But57, Tal79}, we can deform the integration profile over the vertical line to the profile $\Gamma$ where the contour $\Gamma$ is an open piecewise smooth curve running from $-\imagunit\infty$ to $+\imagunit\infty$ surrounding all singularities of $\hat{\out}$ in \eqref{eqn:LT:cf}. In other words, we obtain
\begin{equation}\label{eqn:Bro:int}
	\out(t)\;=\;\frac{1}{2\pi\imagunit}\int_{\Gamma} \ee^{\lapVar t}\fC\hat{\stx}(\lapVar)\;\dz+\fD\inp(t).
\end{equation}
Inserting the expression for $\hat{\stx}(\lapVar)$ from~\eqref{eqn:LT:cf} into~\eqref{eqn:Bro:int}, we obtain
\begin{equation}\label{eqn:Bro:int2}
	\out(t)\;=\;\frac{1}{2\pi\imagunit}\int_{\Gamma} \ee^{\lapVar t}\fC\left(\lapVar\fE-\fA\right)^{-1}\left(\fE\inMat\inSolSub+\fB\hat{\inp}(\lapVar)\right)\;\dz+\fD\inp(t).
\end{equation}
From the system theory language, in~\eqref{eqn:Bro:int2}, the rational function $\fC\left(\lapVar\fE-\fA\right)^{-1} \fB$  is the transfer function of \LCS from the input $\inp(t)$ to the output $\out(t)$; and the rational function $\fC\left(\lapVar\fE-\fA\right)^{-1}\left(\fE\inMat\right)$ is the mapping from the initial condition to the output. These two rational functions fully determine the system behavior. Thus, as in~\cite{BeaGM17}, the output of \LCS is decomposed into the outputs of a \LCS with a zero initial condition and input $\inp$ and a \LCS with an inhomogeneous initial condition with $\inp\equiv \zeroVec$. This will be exploited later on in designing the integration profiles. And as expected these two rational functions and their norms will be heavily used in the analysis below.

Note that, in general, the singularities of the integrand function in \eqref{eqn:LT:cf} coincide with the eigenvalues of the matrix pencil $\lapVar\fE-\fA$ and the poles of $\hat{\inp}$. Moreover, the behaviour of $\|(\lapVar\fE-\fA)^{-1}\|$ in the sectorial domain is also crucial for the construction of the integration profile $\Gamma$, therefore a characterization of it is important to justify the use of any \CIM. 

The evaluation of $\fD\inp(t)$ is independent of $\stateDim$; therefore, the main task is the approximation of the Bromwich integral, i.e., the integral in~\eqref{eqn:Bro:int}. Let us parametrize the
integration contour $\Gamma$ by $\lapVar:\R\rightarrow\Gamma$ for a suitable mapping $\lapVar(\intvar)$ so that we have
\begin{equation}\label{eqn:int:fun}
	\int_{\Gamma}\ee^{\lapVar t}\fC\hat{\stx}(\lapVar)\;\dz\;=\;\int_{\R}\fG(\intvar)\;\dintvar,\quad \fG(\intvar)\;\vcentcolon=\;\ee^{\lapVar(\intvar) t} \fC\hat{\stx}(\lapVar(\intvar))\frac{\partial\lapVar}{\partial \intvar}(\intvar).
\end{equation}
Since we are interested in approximating $\stx(t)$ within precision
$\tol$, we only consider the portion of the Bromwich integral parameterized in $[-c\pi, c\pi]$, that is
\begin{equation*}
	\int_{\R} \fG(\intvar)\;\dintvar\;\approx\; \int_{-c\pi}^{c\pi} \fG(\intvar)\;\dintvar,
\end{equation*}
for a certain truncation parameter $c \in (0, c_{\max})$. Finally, the application of a quadrature formula to approximate \eqref{eqn:Bro:int} provides a numerical approximation of $\out$, for a given time $T$, or even time windows $[T,\Lambda T]$, $\Lambda>1$, without the need to compute it at intermediate time instants. For instance, applying the trapezoidal rule yields the desired approximation $\out_N(T)$ of $\out(T)$ given by
\begin{equation}\label{eqn:quad:app}
	\out_N(T)\;\vcentcolon=\;\fD\inp(T)+\frac{c}{\imagunit N}\sum_{j=1}^{N-1}\fG(\intvar_j), \text{ with }\intvar_j\;=\;-c\pi+j\frac{2c\pi}{N},\quad j=1,\ldots,N-1.
\end{equation}
 Note that, at this stage, no projection-based \MOR has been employed here. Instead, \CIM should be regarded as an alternative to traditional time-integration schemes, such as Runge–Kutta or multistep methods.
\subsection{Construction of \texorpdfstring{$\Gamma$}{TEXT}}\label{subsection:2.1}
For the construction of the integration profile $\Gamma$ we suitably modify the method detailed in \cite[Sec.~5]{Guglielmi2021} for the elliptic profile \cite[Sec.~2.1]{Guglielmi2021}, which is based on the error bound provided in \cite[Thm.~2]{Guglielmi2020}. The \CIMs of \cite{Guglielmi2020,Guglielmi2021} consider the Cauchy problem
\begin{equation}\label{eqn:old:pro}
     \left\{\quad \begin{aligned}
		\dot{\stx}(t)\;  &= \;\fA \stx(t) +\source(t), \\
         \stx(t_0) \;&=\; \stx_0, \\
	\end{aligned}\right.
\end{equation}
and, for fixed $\stx_0$, $\source$, and target accuracy $\tol$, of \cite{Guglielmi2020,Guglielmi2021} show how to construct $\Gamma$ and how many quadrature nodes to chose so that the resulting approximation $\stx_N$ satisfies
\begin{equation}\label{eqn:state:acc}
    \max_{t\in[T,\Lambda T]}\|\stx(t)-\stx_N(t)\|\;\le\;\tol.
\end{equation}
Unlike \eqref{eqn:old:pro}, the \CS \eqref{eqn:LTI:cf} requires the construction of the profile while considering the involvement of matrices $\fB$ (input), $\fC$ (output), $\inMat$ (initial condition), and $\fE$. Furthermore, our objective is to achieve a singular integration profile applicable to any initial condition $\stx_0\in\calF$ and input function $\inp$.

We extend the framework to account for the matrices $\fB$, $\fC$, $\fE$
and $\inMat$ by carefully adapting the arguments of~\cite{Guglielmi2021}.  Moreover, to obtain an integration profile independent of the initial data, we replace~\eqref{eqn:state:acc} with
\begin{equation}\label{eqn:new:cond}
    \max_{t\in[T,\Lambda T]}\frac{\|\out(t)-\out_N(t)\|}{\|\stx_0\|+\max_{\lapVar\in\{\Gamma_{left},\Gamma_{right}\}}\|\hat\inp(\lapVar)\|}\;\le\;\tol,\quad\forall\stx_0\in\calF\quad\text{and}\quad\forall \inp\in \calU,
\end{equation}
where $\calU$ is the set of admissible input functions that we discuss in more detail in \Cref{sec:CIM:ad:inp}.
 
We now recall some details on the parametrization of $\Gamma$ through an elliptic profile and then show how its construction can be adapted to satisfy \eqref{eqn:new:cond}.
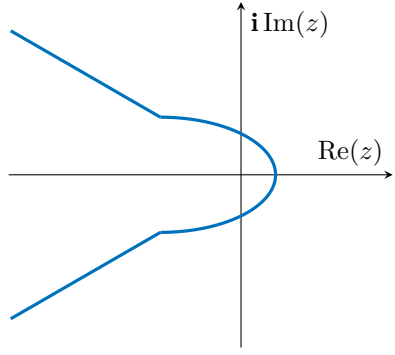
\begin{figure}[t]
\begin{center}
\input{img/Fig1.tex}
\end{center}
\caption{The integration profile $\Gamma$ from \cite{Guglielmi2020}.}
\label{fig:plot_profile1}
\end{figure}
We follow the elliptic parametrization originally proposed in \cite{Guglielmi2020} and subsequently exploited in \cite{Guglielmi2021,GugM23}. The contour $\Gamma$ is given by
\begin{equation}\label{eq:Gamma}
\lapVar_{\Gamma}(\intvar)\;\vcentcolon=\;\left\{\begin{array}{ll}
\ell_1(\intvar), \quad    & \intvar \le -\frac{\pi}{2}, \\[1mm]
\lapVar(\intvar), \quad  & -\frac{\pi}{2} \le \intvar \le \frac{\pi}{2}, \\[1mm]
\ell_2(\intvar), \quad  & \intvar\ge \frac{\pi}{2}, \\
\end{array}\right.
\end{equation}
where, for the constant parameters $A_1, A_2, A_3$ to be determined,
\[
\lapVar(\intvar)\; \vcentcolon=\; A_1\cos \intvar+\imagunit A_2\sin \intvar+A_3
\]
parametrizes an elliptic arc; and
\[
\ell_{1}(\intvar) \;\vcentcolon= \;A_3 + \intvar+\frac{\pi}{2}-\imagunit\left(A_2 - \tilde{d}\left(\intvar+\frac{\pi}{2}\right)\right),\quad
\ell_2(\intvar) \;\vcentcolon=\; A_3-\intvar+\frac{\pi}{2}+\imagunit\left(A_2+\tilde{d}\left(\intvar-\frac{\pi}{2}\right)\right)
\]
parametrize two half-lines; see \Cref{fig:plot_profile1}. The parameter $\tilde{d}$ encodes solely the inclination of the half-line and, as we will discuss shortly, is not relevant for the purposes of our approximation. The choice of parameters $A_1, A_2$ ad $A_3$ is discussed in \cite{Guglielmi2020} and is fundamentally
based on knowledge of the $\varepsilon$-pseudospectrum of $\fA$ in a rectangular region surrounding the rightmost section of the spectrum. The $\varepsilon$-pseudospectrum of $\fA$ is the set of points $\lapVar$ of the complex plane such that $\|(\lapVar\fI_{\stateDim}-\fA)^{-1}\|\ge\frac{1}{\varepsilon}$ holds. However, in the context of the problem \eqref{eqn:LT:cf} it is more relevant to consider $\|\fC(\lapVar\fE-\fA)^{-1}\fE\inMat\|$ and $\|\fC(\lapVar\fE-\fA)^{-1}\fB\|$ over a suitable region of the complex plane. This is perfectly in line with the output formula~\eqref{eqn:Bro:int2} where these two rational (transfer) functions determine the output via the inverse Laplace transform. 

Practically, the only part of the integration contour $\Gamma$~\eqref{eq:Gamma} employed is the elliptical arc expressed through $\lapVar$. To enhance quadrature performance, the domain of $\lapVar$ is expanded to encompass a rectangular domain in the complex plane by
\begin{equation}\label{eq:mapping}
	\lapVar(\intvar+\imagunit \intVarTwo)\;\vcentcolon=\;A_1(\intVarTwo)\cos \intvar+\imagunit A_2(\intVarTwo)\sin \intvar+A_3(\intVarTwo), \quad \intvar\in \left[-\frac{\pi}{2},\frac{\pi}{2}\right],\; \intVarTwo\in[-a,a],
\end{equation}
for a certain parameter $a\in\R^{+}$ to be determined. The complex function \eqref{eq:mapping} is required to be holomorphic in the rectangle
\[
\calR\; \vcentcolon=\; [-\pi/2,\pi/2] \times \imagunit[- a,  a]
\]
and thus the Cauchy-Riemann equations must hold. In this way, one obtains that $A_3$ is necessarily a constant and
\begin{eqnarray}
\label{eq:A1}
A_1(y)&=& a_1\ee^\intVarTwo+a_2\ee^{-\intVarTwo},\\
\label{eq:A2}
A_2(y)&=& a_2\ee^{-\intVarTwo}-a_1\ee^\intVarTwo,
\end{eqnarray}
with $a_1,a_2\in\R$. The resulting mapping is holomorphic throughout the specified rectangular domain, which implies exponential convergence of the trapezoidal quadrature rule when applied to the integral obtained after parametrizing the elliptic contour by $\lapVar(\intvar)$; see \cite[Thm.~2]{Guglielmi2020}.
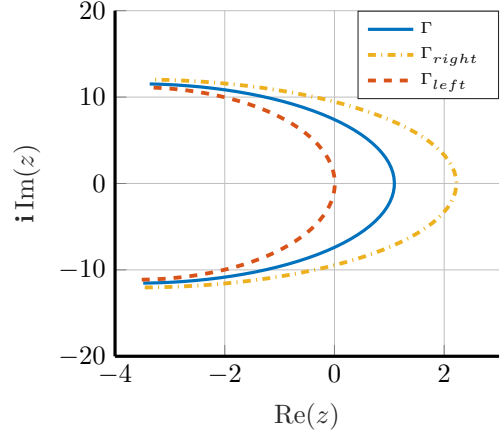
\begin{figure}[t]
\begin{center}
\input{img/Fig2.tex}
\caption{The ellipse $\Gamma_{left}$, the integration profile $\Gamma$, and the ellipse $\Gamma_{right}$. \label{fig_ell}}
\end{center}
\end{figure}
The rectangle $\calR$ is mapped in an elliptical ring-shaped region. In particular, the upper horizontal side of the rectangle is mapped into the inner ellipse $\Gamma_{left}$ in \Cref{fig_ell}, which is selected so that it is external to the sets given by 
\begin{align}\label{eqn:new:pse:set}
\begin{aligned}
   \calP_\varepsilon(\fA,\fC,\fE,\inMat)\;\vcentcolon=&\; \left
   \{\lapVar\in\C\;|\;\|\fC(\lapVar\fE-\fA)^{-1}\fE\inMat\|\ge\frac{1}{\varepsilon}\right\},\\
   \calP_\varepsilon(\fA,\fB,\fC,\fE)\;\vcentcolon=&\; \left
   \{\lapVar\in\C\;|\;\|\fC(\lapVar\fE-\fA)^{-1}\fB\|\ge\frac{1}{\varepsilon}\right\}
   \end{aligned}
\end{align}
for a suitable value of $\varepsilon$. Specifically, we fix the center of the ellipse $\lapVar^L$, its right intersection with the real axis $\lapVar^R$ and an interpolation point ${\lapVar^B}$. The center $\lapVar^L$ is chosen such that ${\ee}^{\lapVar^L T} = {\eps}$ holds, where $\eps$ is the working machine precision and $T$ the time of interest, while $\lapVar^R$ is greater than or equal to the rightmost intersection point with the set \eqref{eqn:new:pse:set} in the real axis. The interpolation point ${\lapVar^B}=d + \imagunit r$ is chosen in such a way that the ellipse encloses \eqref{eqn:new:pse:set} for a suitable choice of $\varepsilon$. Next, the half-elliptic integration profile
\begin{equation*} 
\Gamma:\; \lapVar(x)=(a_1+a_2)\cos \intvar+\imagunit(a_2-a_1)\sin \intvar+A_3,
\end{equation*}
is determined, with coefficients $a_1$, $a_2$, $A_3$ depending on the unique free parameter $a$; see \cite{Guglielmi2020} for more details. Indeed, imposing the ellipse $\Gamma_{left}$ to be centered on $\lapVar^L$, and to pass through the points $z^R$, and $\lapVar^B=d+\imagunit r$, we get
\begin{equation}\label{eq1}
a_1\ee^a+a_2\ee^{-a}=\lapVar^R-\lapVar^L,\quad
a_2\ee^{-a}-a_1\ee^a=\frac{r}{\sin(\theta)},\quad
A_3=\lapVar^L,
\end{equation}
where
\begin{equation*} 
\theta=\arccos\left(\frac{d-z^L}{z^R-z^L}\right)\,.
\end{equation*}
Solving \eqref{eq1} for $a_1,a_2, A_3$ yields
\begin{align*}
\begin{aligned}
a_1\;=\; \frac{\ee^{-a}}{2}\left(\lapVar^R-\lapVar^L-\frac{r}{\sin(\theta)}\right), 
\quad a_2 \; = \; \frac{\ee^{a}}{2}\left(\lapVar^R-\lapVar^L+\frac{r} {\sin(\theta)}\right),\quad
A_3 \; = \; \lapVar^L;
\end{aligned}
\end{align*}
which depend only on the real parameter $a$. By replacing $\|(\lapVar\fI_{\stateDim}-\fA)^{-1}\|$ with $\|\fC(\lapVar\fE-\fA)^{-1}\fE\inMat\|$ and $\|\fC(\lapVar\fE-\fA)^{-1}\fB\|$, we adopted the procedure detailed in \cite[Sec.~3]{Guglielmi2020}
for the construction of $\Gamma_{left}$, thus bringing the input-output perspective from systems theory into the \CIM framework.

 Once $\Gamma_{left}$ is fixed, it only remains to determine the parameter $a$, this is done based on the error estimate stated in \cite[Thm.~2]{Guglielmi2020}. In particular, the goal is to minimize the number of quadrature points employed for a fixed specified accuracy $\tol$. The solution of this problem provides the parameter $a$ and the estimated number of quadrature points $\tilde{N}$ for which \eqref{eqn:state:acc} is valid. As already discussed, in \cite{Guglielmi2020,Guglielmi2021} the procedure is stated for fixed $\stx_0$ and $\source$ then, $\Gamma$ is determined. Here we suitably modify this approach to achieve \eqref{eqn:new:cond}. From \cite[Thm.~2, Sec.~3.2]{Guglielmi2020} we can write
\begin{align}
\calE(N)\;\vcentcolon=&\;\max_{t\in[T,\Lambda T]}\frac{\|\out(t)-\out_N(t)\|}{\|\inSolSub\|+\max_{\lapVar\in\{\Gamma_{left},\Gamma_{right}\}}\|\hat\inp(\lapVar)\|}\;\label{eqn:err:def}\\
        \le&\;\frac{1}{\|\inSolSub\|+\max_{\lapVar\in\{\Gamma_{left},\Gamma_{right}\}}\|\hat\inp(\lapVar)\|}\frac{2\pi c(a)M_{right}(a)+\pi M_{left}}{\ee^{\tfrac{a}{c(a)}N}-1},\label{eqn:err:bound}\\
       \text{where}\quad M_{left}\;\vcentcolon=&\;\frac{1}{2\pi} \max_{\lapVar\in\Gamma_{left}}\Big\|\ee^{\lapVar\phi(\lapVar)}\fC\left(\lapVar\fE-\fA\right)^{-1}\left(\fE\inMat\inSolSub+\fB{\hat \inp(\lapVar)}\right)\lapVar'\Big\|~\mbox{and}~\label{eqn:def:const1}\\
       M_{right}(a)\;\vcentcolon=&\;\frac{1}{2\pi} \max_{\lapVar\in\Gamma_{right}(a)}\Big\|\ee^{\lapVar\phi(\lapVar)}\fC\left(\lapVar\fE-\fA\right)^{-1}\left(\fE\inMat\inSolSub+\fB{\hat \inp(\lapVar)}\right)\lapVar'\Big\|,\label{eqn:def:const2}
\end{align}
with
\begin{align*}
     \phi(\lapVar)\;\vcentcolon=&\;\left\{\begin{array}{ll}
T, \quad    & \real(\lapVar) < 0, \\[1mm]
0, \quad  & \real(\lapVar) = 0, \\[1mm]
\Lambda T, \quad    & \real(\lapVar) > 0, \\
\end{array}\right..
\end{align*}
We note that $M_{left}$ in~\eqref{eqn:def:const1} and \eqref{eqn:def:const2} can be bounded by 
\begin{align}
       M_{left}\;\le&\;\widetilde{M}_{1,left}\|\inSolSub\|
+\widetilde{M}_{2,left}\max_{\lapVar\in\Gamma_{left}}\|{\hat \inp(\lapVar)}\|,\label{eqn:def:const:12}
\end{align}
where
\begin{align*}
\widetilde{M}_{1,left}\;\vcentcolon=&\;\frac{1}{2\pi} \max_{\lapVar\in\Gamma_{left}}\Big\|\ee^{\lapVar\phi(\lapVar)}\fC\left(\lapVar\fE-\fA\right)^{-1}\fE\inMat\lapVar'\Big\|~\mbox{and}\\
\widetilde{M}_{2,left} \;\vcentcolon=\;&\frac{1}{2\pi} \max_{\lapVar\in\Gamma_{left}}\Big\|\ee^{\lapVar\phi(\lapVar)}\fC\left(\lapVar\fE-\fA\right)^{-1}\fB\lapVar'\Big\|.
\end{align*}
Similarly $M_{right}$ in~\eqref{eqn:def:const2} is bounded by 
\begin{align}
         M_{right}(a)\;\le&\;\widetilde{M}_{1,right}(a)\|\inSolSub\|
+\widetilde{M}_{2,right}(a) \max_{\lapVar\in\Gamma_{right}(a)}\|{\hat \inp(\lapVar)}\|,\label{eqn:def:const:22}
\end{align}
where 
\begin{align*}
\widetilde{M}_{1,right}(a)\;\vcentcolon=&\;\frac{1}{2\pi} \max_{\lapVar\in\Gamma_{right}(a)}\Big\|\ee^{\lapVar\phi(\lapVar)}\fC\left(\lapVar\fE-\fA\right)^{-1}\fE\inMat\lapVar'\Big\|~\mbox{and}~\\
\widetilde{M}_{2,right}(a) \;\vcentcolon=\;&\frac{1}{2\pi} \max_{\lapVar\in\Gamma_{right}(a)}\Big\|\ee^{\lapVar\phi(\lapVar)}\fC\left(\lapVar\fE-\fA\right)^{-1}\fB\lapVar'\Big\|.
\end{align*}
Substituting \eqref{eqn:def:const:12} and \eqref{eqn:def:const:22} into \eqref{eqn:err:bound} we get
\begin{align}\label{eqn:err:est}
    \begin{aligned}
         \calE(N)\;\le&\;\frac{2\pi c(a)(\widetilde M_{1,right}(a)+\widetilde M_{2,right}(a))+\pi (\widetilde M_{1,left}+\widetilde M_{2,left})}{\ee^{\tfrac{a}{c(a)}N}-1}.
    \end{aligned}
\end{align}
Now, if we set the right-hand side of \eqref{eqn:err:est} equal to the desired accuracy $\tol$ and express $N$ as a function of $a$, we get
\begin{align*}
    \begin{aligned}
        N\;=&\;\frac{c(a)}{a}\left(\log\left(\tol+q(a)\right)-\log\left(\tol\right)\right),\\
        q(a)\;\vcentcolon=&\;2\pi c(a)(\widetilde M_{1,right}(a)+\widetilde M_{2,right}(a))+\pi (\widetilde M_{1,left}+\widetilde M_{2,left});
    \end{aligned}
\end{align*}
thus, we can compute $\tilde N$ by asking that the number of quadrature points to reach $\tol$ is as small as possible, i.e., 
\begin{align}\label{eqn:opt:num:QP}
    \begin{aligned}
        \tilde N\;\vcentcolon=\;\min_{a\in[0,a_{\max}]} \frac{c(a)}{a}\left(\log\left(\tol+q(a)\right)-\log\left(\tol\right)\right),
    \end{aligned}
\end{align}
where $a_{\max}$ is determined based on the value of $\ee^{\lapVar(0+\imagunit a_{\max})\Lambda T}$, that is, based on the location of the vertex of $\Gamma_{right}(a_{\max})$ in \Cref{fig_ell}. We refer to the discussion in \cite[Sec.~5]{Guglielmi2021} for more details on the choice of parameters and the optimization procedure to solve \eqref{eqn:opt:num:QP}.
\subsection{About the computational cost}\label{sec2:comp-cost}
As a standard in projection-based \MOR we distinguish the computational cost for the offline and online phases. 
\begin{itemize}
    \item \textbf{Offline phase:} The first step, for given a $\tol$, is the construction of the integration profile $\Gamma$ such that \eqref{eqn:new:cond} holds. This requires computing the largest singular value of the matrices $\fC(\lapVar\fE-\fA)^{-1}\fE\inSolSub$ and $\fC(\lapVar\fE-\fA)^{-1}\fB$ in a suitable region of the complex plane. In other words, the goal is to find the maximum value of these two transfer functions over the region of interest. 
    Then \eqref{eqn:opt:num:QP} has to be solved and the main computational burden is the estimation of the constants \eqref{eqn:def:const2}, which involves solving eigenvalue problems of dimension $\stateDim$. Once $\Gamma$ is determined, a second stage of the offline phase of \CIM for \CS consists of computing (eventually exploiting parallelization) and storing the matrices
    \begin{subequations}\label{eqn:small:size:mat}
        \begin{align}
            \fK_{1,j}\;=&\;\frac{\partial\lapVar}{\partial\intvar}(\intvar_j)\fC\left(\lapVar(\intvar_j)\fE-\fA\right)^{-1}\fE\inMat\in\C^{\outDim\times \stateDimRed_{\calF}},\quad j=1,\ldots,\tilde N-1,\label{eqn:K1:x0}\\ 
            \fK_{2,j}\;=&\;\frac{\partial\lapVar}{\partial\intvar}(\intvar_j)\fC\left(\lapVar(\intvar_j)\fE-\fA\right)^{-1}\fB\in\C^{\outDim\times\inpDim},\quad j=1,\ldots,\tilde N-1.\label{eqn:K2:u}
        \end{align}
    \end{subequations}
    Note that neither \eqref{eqn:K1:x0} nor \eqref{eqn:K2:u} depend on $\stateDim$ in terms of storage costs.
    \item \textbf{Online phase:} The output $\out(t)$, for $t\in[T,\Lambda T]$, can be quickly evaluated for any input function $\inp \in\calU$ and any $\inSolSub$ using the precomuted small-size matrices \eqref{eqn:small:size:mat}. Indeed, from \eqref{eqn:quad:app} we can write
    \begin{align*}
        \begin{aligned}
            \out_{\tilde N}(t)\;=\;\fD\inp(t)+\frac{c}{\imagunit \tilde N}\sum_{j=1}^{\tilde N-1} \ee^{\lapVar(\intvar_j)t}\left(\fK_{1,j}\inSolSub+\fK_{2,j}\hat\inp(\lapVar(\intvar_j))\right), \quad\intvar_j\;=\;-c\pi+j\frac{2c\pi}{\tilde N}.
        \end{aligned}
    \end{align*}
Since this evaluation relies entirely on the precomputed, small-size matrices in~\eqref{eqn:small:size:mat}, its cost is independent of the state dimension $\stateDim$.
\end{itemize}
We conclude with the following observations. If the integrand is conjugate symmetric and $\Gamma$ is symmetric with respect to the real axis, the number of addends, and thus the number of stored matrices to be computed, can be halved. Finally, despite the use of a simple trapezoidal quadrature rule, the error in the spectral norm between $\out_N(T)$ and $\out(T)$ decays exponentially with respect to the number of quadrature points due to the regularity properties of the integrand function; see \eqref{eqn:err:bound} and \cite{TreW14} for more details on the exponentially convergent trapezoidal quadrature rule. This means that high accuracy over the time window can be reached with only a few quadrature points, requiring few matrices to be precomputed and stored.

\subsection{Discussion on the efficency of \CIM and admissible input functions.}\label{sec:CIM:ad:inp}

The \CIM approach is grounded in two fundamental assumptions that we now recall: firstly, the singularities of the integrand function in \eqref{eqn:inv:Lap} reside within a sectorial area of the complex plane where $\real(\lapVar)< \gamma$; and secondly, the integrand function in \eqref{eqn:inv:Lap} decays as $|\lapVar|\rightarrow\infty$. As discussed in \Cref{sec2}, the singularities of the integrand function are characterized by the generalized eigenvalues of the matrix pair $(\fA,\fE)$ and the poles of $\hat{\inp}(\lapVar)$. Specifically, situations where these eigenvalues exhibit large imaginary components with minimal real components, or even zero real component when $\fE^{-1}\fA$ is skew-symmetric pose challenges for any \CIM. This is because such conditions cause $\Gamma$ to almost align vertically before deviating to the left in the complex plane. Consequently, numerous quadrature points are required in the region where $\Gamma$ is nearly vertical, significantly increasing computational costs. Consider the scenario of discretizing the transport equation as an illustrative case. By applying central finite difference methods and appropriate boundary conditions, one typically arrives at a skew-symmetric matrix whose eigenvalues possess an imaginary component that is proportional to the inverse of the spatial discretization step. Recently, efforts to partially alleviate this issue include leveraging the regularity of initial data to enhance the convergence rate of the quadrature method, as discussed in \cite{HorG24}, or employing generalized convolution quadratures \cite{GuoL25}.

A challenging category of issues for \CIM when addressing the equation $\fE\dot \stx=\fA\stx$ arises from the presence of a highly non-normal matrix $\fE^{-1}\fA$. It is well established that, in such scenarios, the values of $\|(\lapVar\fE-\fA)^{-1}\|$ can become exceedingly large, even when distant from the generalized spectrum. A significant resolvent norm induces instability in any \CIM; hence, matrices whose spectra are entirely located on the real axis can still pose difficulties if they exhibit high non-normality. Nevertheless, within the context of \LCS with a zero initial solution, or when initial solutions reside in a low-dimensional subspace defined by the columns of a tall rectangular matrix $\inMat$, the focal quantities are
\begin{equation}\label{eqn:mod:res}
    \|\fC(\lapVar\fE-\fA)^{-1}\fB\|\quad\quad  \text{and}\quad\quad\|\fC(\lapVar\fE-\fA)^{-1}\fE\inMat\|,
\end{equation}
which could be significantly smaller than $\|(\lapVar\fE-\fA)^{-1}\|$, thus facilitating efficient \CIM-\LCS, even in the presence of strong non-normal behavior of $\fE^{-1}\fA$. 
This is due to the input-output based system theory perspective on \MOR: one is not necessarily interested in the whole state $\stx(t)$, but the output $\out(t)$ due to the input $\inp(t)$ and the initial condition $\stx_0$.
Let us further analyze the scenario where one of the two norms in \eqref{eqn:mod:res} is significantly larger than the other for $\lapVar\in\Gamma$. In this case, it would be beneficial to design separate integration profiles to deal with the initial solution and the input function, according to the findings of \cite{BeaGM17} where \MOR for \LCS with initial conditions were considered. Specifically, the output of \LCS can be decomposed into the outputs of a \LCS with a zero initial solution and input $\inp$ and a \LCS with an inhomogeneous initial condition while maintaining $\inp\equiv \zeroVec$. This strategy would effectively reduce the total number of necessary linear system solutions.

Regarding $\inp$, it is crucial to consider the location of singularities in potential input functions during the construction of $\Gamma$. Assume that all generalized eigenvalues are real and negative and that $\|(\lapVar\fE-\fA)^{-1}\|$ becomes significantly large only ``near'' these eigenvalues. In such cases, it is preferable for $\Gamma$ to be positioned as near as possible to the real axis to reduce the number of quadrature points needed for achieving a desired accuracy $\tol$. Nevertheless, this positioning might exclude input functions with poles possessing a sufficiently large imaginary part, such as $\hat{\inp}=\lapVar/(\lapVar^2+\alpha^2)$, which equates to the Laplace transform of $\inp(t)=\cos(\alpha t)$. Concerning the second assumption, $\|(\lapVar\fE-\fA)^{-1}\|$ asymptotically decreases as $\lapVar^{-1}$. Consequently, any input function satisfying $\hat{\inp}(\lapVar)= \calO(\lapVar^\nu)$, with $\nu<1$ and for $|\lapVar|\rightarrow\infty$, remains technically permissible, as long as the contour accommodates its poles.

We conclude by examining the use of the Heaviside function (also known as the step function) as input. If the discontinuity occurs at a time $t_0<T$, the Laplace transform of the input is given by $\hat \inp(\lapVar) =\frac{\ee^{-\lapVar t_0}}{\lapVar}$. This is compatible with the two assumptions, as long as zero is counted as a pole. Nevertheless, it reduces the exponential decay rate since the exponent in \eqref{eqn:inv:Lap} becomes $\ee^{\lapVar (T-t_0)}$. Thus, it affects the decay properties of the integrand function in $\Gamma$, influencing the choice of the hyperparameter $z_L$. The key point is that discontinuous input functions can still be handled by \CIM, but the construction of $\Gamma$ must be adapted to reflect these characteristics of the input signal.

\section{A non-linear \MOR method for parametric control systems using \CIM}\label{sec3}

This section develops a nonlinear projection-based \MOR method for parametric \CS of type \eqref{eqn:LTI:cf:par}, extending the \CIM-based framework of \cite{GugM23}. Unlike \cite{GugM23}, our focus is on approximating a specific quantity of interest, namely the output $\out$, rather than the full state; this shift leads us to adopt a Petrov--Galerkin projection, rather than a simpler Galerkin projection, in constructing the \ROM. Moreover, the projection spaces are constructed so as to be independent of any initial condition in a given subspace $\inSolSub$ and of any input function belonging to the set of admissible functions. This property constitutes a notable distinction from the approach developed in\cite{GugM23}.

There are three assumptions that we must discuss to have an efficient \MOR procedure. The first is that, for the desired precision $\tol$, we are able to identify a unique profile $\Gamma$ and a unique number of quadrature points $N$ such that 
\begin{align*}
   \calE(\prmtr,\inp,\stx_0)\vcentcolon=\max_{t\in[T,\Lambda T]}\frac{ \|\out(t,\prmtr)-\out_{N}(t,\prmtr)\|}{\|\stx_0\|+\max_{\lapVar\in\{\Gamma_{left},\Gamma_{right}\}}\|\hat\inp(\lapVar)\|}
   \le\tol,\text{ for all }\prmtr\in\prmtrSet,\;\stx_0\in\calF,\;\inp\in \calU;
\end{align*}
where $\out_{N}(t,\prmtr)$ is the approximation of $\out(t,\prmtr)$ through \CIM, i.e.,
\begin{align}
    \out_N(t,\prmtr)\;\vcentcolon=&\;\fD(\prmtr)\inp(t)+\frac{c}{\imagunit N}\sum_{j=1}^{N-1}\ee^{\lapVar(\intvar_j)t}\fC(\prmtr)\hat\stx(\lapVar(\intvar_j),\prmtr)\frac{\partial\lapVar}{\partial \intvar}(\intvar_j),\label{eqn:out:par:CIM}\\
    \hat \stx(\lapVar(\intvar),\prmtr)\;\vcentcolon=&\;\left(\lapVar(\intvar)\fE(\prmtr)-\fA(\prmtr)\right)^{-1}\left(\fE(\prmtr)\stx_0+\fB(\prmtr)\hat\inp(\lapVar(\intvar))\right)\label{eqn:stx:par:CIM}.
\end{align}
Note that this directly implies that $\lapVar(\intvar_j)\fE(\prmtr)-\fA(\prmtr)$ has to be invertible for all $\prmtr\in\prmtrSet$ and $j=1,\ldots,N-1$. A possible methodology to assess the validity of this assumption is addressed in \cite[Sec.~3.5]{GugM23}. Shortly, the challenge arises when the variation of the parameters within $\prmtrSet$ leads to significant fluctuations in the location of the singularities of the integrand function and in the norm $\|(\lapVar(\intvar_j)\fE(\prmtr)-\fA(\prmtr))^{-1}\|$.

The second assumption is a classical one for \MOR of parametric problems: the affine dependence of the system matrices on the parameters. Thus we assume we can write 
\begin{align}
\begin{aligned}\label{eqn:CS:par:mat}
        \fE(\prmtr)\;=&\;\sum_{i=1}^{Q_{\fE}}\epsilon_i(\prmtr)\fE_i,\quad \;\;\fA(\prmtr)\;=\;\sum_{j=1}^{Q_{\fA}}\alpha_j(\prmtr)\fA_j,\\
\fB(\prmtr)\;=&\;\sum_{k=1}^{Q_{\fB}}\beta_k(\prmtr)\fB_k,\quad\fC(\prmtr)\;=\;\sum_{l=1}^{Q_{\fC}}\gamma_l(\prmtr)\fC_l
\end{aligned}
\end{align}
for some analytic functions $\epsilon_i$, $\alpha_j$, $\beta_k$, $\gamma_l:\prmtrSet\rightarrow\R$, constant matrices $\fE_i$, $\fA_j$, $\fB_k$, $\fC_l$, and positive integers $Q_{\fE}$, $Q_{\fA}$, $Q_{\fB}$, $Q_{\fC}\ll\stateDim$.

The third assumption was already discussed at the beginning of \Cref{sec2}. The subspace to which $\stx_0$ belongs must have a dimension of $\stateDimRed_{\calF}$, with $\stateDimRed_{\calF}\ll\stateDim$. Thus, $\stx_0$ is expressed as $\stx_0=\inMat\inSolSub$, where $\inSolSub\in\R^{\stateDimRed_{\calF}}$ and $\inMat\in\R^{\stateDim\times \stateDimRed_{\calF}}$ are a matrix with orthonormal columns spanning the subspace $\calF$. This assumption has been commonly used in previous works on \MOR~ for  \LCS with inhomogeneous initial conditions, see, e.g., \cite{BeaGM17,HeiRA11,BauBF14,SchV23}. 
\subsection{From the \FOM to the \ROM: the \CIM-\MOR for linear parametric control systems}
To ensure brevity in this section, we utilize the shorthand notation $\lapVar_j\equiv\lapVar(\intvar_j)$. The fundamental concept behind projection-based \ROM using \CIM involves the estimation of $\hat \stx(\lapVar_j,\prmtr)$ by employing $\hat{\fq}(\lapVar_j;\prmtr)\vcentcolon = \fV_j\hat \stx_{\mathrm{r}}(\lapVar_j;\prmtr)$, with $\hat \stx_{\mathrm{r}}(\lapVar_j;\prmtr)\in \C^{\stateDimRed_j}$ and $\fV_j \in\C^{\stateDim\times\stateDimRed_j}$, with $\fV_j$ having orthonormal columns and $\stateDimRed_j\ll\stateDim$ for each $j=1,\ldots,N-1$. We therefore plug $\hat{\fq}(\lapVar_j;\prmtr)$ in place of $\hat{\stx}(\lapVar_j;\prmtr)$ into \eqref{eqn:stx:par:CIM}, which yields the residual
\begin{equation}\label{eqn:red:par:CIM}
    \res(\lapVar_j,\prmtr)\;\vcentcolon=\;(\lapVar_j\fE(\prmtr)-\fA(\prmtr))\fV_j\hat \stx_{\mathrm{r}}(\lapVar_j,\prmtr)-\fE(\prmtr)\inMat\inSolSub-\fB(\prmtr)\hat{\inp}(\lapVar_j)
\end{equation}
Then, we impose the residual to be orthogonal to the subspace generated by the orthonormal columns of a certain matrix $\fW_j\in\C^{\stateDim\times\stateDimRed_j}$, i.e., $\fW_j^{*}\res(\lapVar_j;\prmtr)=\zeroVec$, which gives
\begin{equation}\label{eqn:rstate:par:CIM}
   (\lapVar_j\fE_{\rr,j}(\prmtr)-\fA_{\rr,j}(\prmtr))\hat\stx_{\mathrm{r}}(\lapVar_j,\prmtr)\;=\;\fX_{\stateDimRed,j}(\prmtr)\inSolSub+\fB_{\rr,j}(\prmtr)\hat\inp(\lapVar_j)
\end{equation}
with
\begin{align}\label{eqn:proj:mat}
\begin{aligned}
\fE_{\rr,j}(\prmtr)\;\vcentcolon=&\;\;\fW_j^{*}\fE(\prmtr)\fV_j,\qquad\fA_{\rr,j}(\prmtr)\;\vcentcolon=&\fW_j^{*}\fA(\prmtr)\fV_j,\\
\;\fB_{\rr,j}(\prmtr)\;\vcentcolon=&\;\;\fW_j^{*}\fB(\prmtr),\qquad\;\;\, \fX_{\stateDimRed,j}(\prmtr)\;\vcentcolon=&\fW^{*}_{j}\fE(\prmtr)\inMat.
\end{aligned}
\end{align}
This is the so-called Petrov-Galerkin projection; if $\fW_j=\fV_j$ we have the standard Galerkin projection. The approximated output is finally given by
\begin{equation}\label{eqn:out:par:CIM:red}
     \out_{\mathrm{r},N}(t,\prmtr)\;\vcentcolon=\;\fD(\prmtr)\inp(t)+\frac{c}{\imagunit N}\sum_{j=1}^{N-1}\ee^{\lapVar_j t}\fC_{\mathrm{r},j}(\prmtr)\hat\stx_{\mathrm{r}}(\lapVar_j,\prmtr)\frac{\partial\lapVar}{\partial \intvar}(\intvar_j),\quad \fC_{\mathrm{r},j}(\prmtr)\;\vcentcolon=\;\fC(\prmtr)\fV_j.
\end{equation}
Note that the affine decomposition assumption \eqref{eqn:CS:par:mat} and the restriction of the initial condition to a small dimensional subspace ensure that, for each istance of $\prmtr\in\prmtrSet$, the matrices in \eqref{eqn:proj:mat} and \eqref{eqn:out:par:CIM:red} can be formed efficiently, i.e., with a computational cost independent of $\stateDim$; see, e.g., \cite[Sec.~3.3]{HesRS16}, \cite[Sec.~3.3]{BenGW15}. We emphasize that due to the \CIM framework we employ to approximate the output at a specific time interval, the reduced order quantities in~\eqref{eqn:proj:mat} vary with the quadrature nodes (frequency sampling points). In other words, there is not a single parametric reduced \LCS. The reduced state-space quantities  vary with the quadarute node. 

We conclude this section by introducing the adjoint problem, i.e.,
\begin{equation}\label{eqn:adj:par:CIM}
    \left(\lapVar_j\fE(\prmtr)-\fA(\prmtr)\right)^{*}\hat \fp(\lapVar_j,\prmtr)=\fC(\prmtr)^{*}.
\end{equation}
Assuming $\hat \fp(\lapVar_j,\prmtr)\in\C^{\stateDim\times\outDim}$ is well approximated by $\fW_j\hat \fp_{\rr}(\lapVar_j,\prmtr)$, we define the so-called dual residual 
\begin{equation}\label{eqn:dual:res}
    \res_{\hat \fp}(\lapVar_j,\prmtr)\;\vcentcolon=\;(\lapVar_j\fE(\prmtr)-\fA(\prmtr))^{\ast}\fW_j\hat \fp_{\rr}(\lapVar_j,\prmtr)-\fC(\prmtr)^{*},
\end{equation}
and the associated reduced order problem 
\begin{equation}\label{eqn:CIM:red:adj}
 (\lapVar_j\fE_{\rr,j}(\prmtr)-\fA_{\rr,j}(\prmtr))^{*}\hat\fp_{\mathrm{r}}(\lapVar_j,\prmtr)\;=\;\fC_{\rr,j}(\prmtr)^{*},
\end{equation}
obtained from \eqref{eqn:dual:res} by imposing the orthogonality condition $\fV_j^{*}\res_{\hat \fp}(\lapVar_j,\prmtr)=\zeroVec$.
\subsection{Greedy spaces generation}
In this section, we discuss how the subspaces $\fV_j$ and $\fW_j$ are constructed for $j=1,\ldots,N-1$. To achieve this, we utilize the so-called greedy algorithm, which is essentially an iterative process in which, during each iteration, one or more vectors of dimension $\stateDim$ are incorporated into the approximation spaces. This process aims to potentially enhance the overall approximation capacity of the basis set, necessitating only a few costly evaluations at each step. A crucial aspect of the greedy algorithm is the presence of an error estimate $\Delta(\prmtr)$, which anticipates the error resulting from the \MOR, specifically the error introduced by substituting the \FOM with size $\stateDim$ by the \ROM with size $\stateDimRed$. This method is referred to as a weak-greedy algorithm to distinguish it from the strong-greedy algorithm, which utilizes direct error measurement. In a strong greedy algorithm, it is necessary to assess the \FOM for every parameter instance within the parametric domain, which conflicts with the \MOR principle. Therefore, it is crucial that the error estimator $\Delta(\prmtr)$ can be computed efficiently for every parameter $\prmtr \in \prmtrSet$. An important characteristic of this technique is its ability to form subspaces that maintain the Kolmogorov $n$-width decay type associated with the original problem. This aspect is examined in \cite{BinCDDPW11}.

The next lemma provides the error estimate that we will employ to construct $\fV_j$ and $\fW_j$ via the weak-greedy algorithm.

\begin{lemma}[Upper error bound for the \CIM-\ROM]\label{lemma:err:est}
    For all $\inSolSub \in\R^{\stateDimRed_{\calF}}$, $\inp\in\calU$, and $\prmtr\in\prmtrSet$ it holds that
\begin{equation*}
     \calE_{\rr}(\prmtr,\inp,\inMat)\;\le\;\Delta(\prmtr)
\end{equation*}
    with
    \begin{subequations}
        \begin{align}
             \calE_{\rr}(\prmtr,\inp,\inMat)\vcentcolon=&\max_{t\in[T,\Lambda T]}\frac{ \|\out_N(t,\prmtr)-\out_{\rr,N}(t,\prmtr)\|}{\|\inSolSub\|+\max_{j=1,\ldots,N-1}\|\hat\inp(\lapVar_j)\|},\label{eqn:err:red:par:CIM}\\
             \Delta(\prmtr)\vcentcolon=&\sum_{j=1}^{N-1}w_j\|(\lapVar_j\fE(\prmtr)-\fA(\prmtr))^{-1}\|\|\res_{\hat \fp}(\lapVar_j,\prmtr)\| \left(\|\res_{\inSolSub}(\lapVar_j,\prmtr)\|+ \|\res_{\inp}(\lapVar_j,\prmtr)\|\right)\label{eqn:err:est:red:par:CIM}
        \end{align}
    \end{subequations}
    where
    \begin{subequations}
        \begin{align}
            w_j\;\vcentcolon=&\;\frac{c}{N}\ee^{\real(z_j)t_j}\Bigg|\frac{\partial\lapVar}{\partial \intvar}(\intvar_j)\Bigg|,\quad t_j=\begin{cases}
                \Lambda T, & \mbox{if }\real(\lapVar_j)\ge 0\\
                 T, & \mbox{if }\real(\lapVar_j)<0
            \end{cases},\quad\inMat(\prmtr)\;\vcentcolon=\;\fE(\prmtr)\inMat,\label{eqn:weights:CIM}\\
             \res_{\inSolSub}(\lapVar_j,\prmtr)\;\vcentcolon=&\;(\lapVar_j\fE(\prmtr)-\fA(\prmtr))\fV_j(\lapVar_j\fE_{\rr,j}(\prmtr)-\fA_{\rr,j}(\prmtr))^{-1}\fW^{*}_j\inMat(\prmtr)-\inMat(\prmtr),\label{eqn:res:ini:par:CIM}\\
          \res_{\inp}(\lapVar_j,\prmtr)\;\vcentcolon=&\;(\lapVar_j\fE(\prmtr)-\fA(\prmtr))\fV_j(\lapVar_j\fE_{\rr,j}(\prmtr)-\fA_{\rr,j}(\prmtr))^{-1}\fW^{*}_j\fB(\prmtr)-\fB(\prmtr),\label{eqn:res:inp:par:CIM}
        \end{align}
    \end{subequations}
    and $\inMat$ is the orthonormal matrix whose columns span the subspace of the admissible initial solutions.
    \end{lemma}
\begin{proof}
    Let us start by bounding the simpler output error. From definitions \eqref{eqn:out:par:CIM}, \eqref{eqn:out:par:CIM:red}, \eqref{eqn:stx:par:CIM}, \eqref{eqn:red:par:CIM}, \eqref{eqn:weights:CIM} and the adjoint system \eqref{eqn:adj:par:CIM} we have
    \begin{align*}
    \begin{aligned}
           \max_{t\in[T,\Lambda T]} \|\out_N(t,\prmtr)-\out_{\rr,N}(t,\prmtr)\|\;\le&\;\sum_{j=1}^{N-1}\frac{c}{N}\ee^{\real(\lapVar_j) t_j}\left|\frac{\partial\lapVar}{\partial \intvar}(\intvar_j)\right|\left\|\fC(\prmtr)\left(\hat\stx(\lapVar_j,\prmtr)-\fV_j\hat\stx_{\mathrm{r}}(\lapVar_j,\prmtr)\right)\right\|\\
           =&\;\sum_{j=1}^{N-1}w_j\left\|\fC(\prmtr)\left(\lapVar_j\fE(\prmtr)-\fA(\prmtr)\right)^{-1}\res(\lapVar_j,\prmtr)\right\|\\
           =&\;\sum_{j=1}^{N-1}w_j\left\|\hat \fp(\lapVar_j,\prmtr)^{*}\res(\lapVar_j,\prmtr)\right\|.
    \end{aligned}
    \end{align*}
Add subtract to $\hat \fp(\lapVar_j,\prmtr)$ its reduced counterpart $\fW_j\hat \fp_{\rr}(\lapVar_j,\prmtr)$ to write
\begin{align*}
    \begin{aligned}
      \sum_{j=1}^{N-1}w_j\left\|\hat \fp(\lapVar_j,\prmtr)^{*}\res(\lapVar_j,\prmtr)\right\|\;\le&\;\sum_{j=1}^{N-1}w_j\left\|\hat \fp_{\rr}(\lapVar_j,\prmtr)^{*}\fW_j^{*}\res(\lapVar_j,\prmtr)\right\|\\
      +&\;\sum_{j=1}^{N-1}w_j\left\|\left(\hat \fp(\lapVar_j,\prmtr)-\fW_j\hat \fp_{\rr}(\lapVar_j,\prmtr)\right)^{*}\res(\lapVar_j,\prmtr)\right\|\\
      =&\;\sum_{j=1}^{N-1}w_j\|\res_{\hat \fp}(\lapVar_j,\prmtr)^{*}\left(\lapVar_j\fE(\prmtr)-\fA(\prmtr)\right)^{-*} \res(\lapVar_j,\prmtr)\|\\
      \le&\;\sum_{j=1}^{N-1}w_j\left\|\left(\lapVar_j\fE(\prmtr)-\fA(\prmtr)\right)^{-1}\right\|\|\res_{\hat \fp}(\lapVar_j,\prmtr)\| \|\res(\lapVar_j,\prmtr)\|,
    \end{aligned}
\end{align*}
where we used the Petrov-Galerkin orthogonality condition $\fW_j^{*}\res(\lapVar_j,\prmtr)=\zeroVec$ and the dual residual definition \eqref{eqn:dual:res}. Concerning $\|\res(\lapVar_j,\prmtr)\|$, using \eqref{eqn:rstate:par:CIM} and the triangular inequality, we get
\begin{align*}
    \begin{aligned}
        \|\res(\lapVar_j,\prmtr)\|\;=&\;\left\|\left(\lapVar_j\fE(\prmtr)-\fA(\prmtr)\right)\fV_j\hat \stx_{\mathrm{r}}(\lapVar_j,\prmtr)-\inMat(\prmtr)\inSolSub-\fB(\prmtr)\hat \inp(\lapVar_j)\right\|\\
        \le&\;\left\|\left(\lapVar_j\fE(\prmtr)-\fA(\prmtr)\right)\fV_j\left(\lapVar_j\fE_{\rr,j}(\prmtr)-\fA_{\rr,j}(\prmtr)\right)^{-1}\fW^{*}_{j}\inMat(\prmtr)-\inMat(\prmtr)\right\|\|\inSolSub\|\\
        +&\;\left\|\left(\lapVar_j\fE(\prmtr)-\fA(\prmtr)\right)\fV_j\left(\lapVar_j\fE_{\rr,j}(\prmtr)-\fA_{\rr,j}(\prmtr)\right)^{-1}\fW^{*}_{j}\fB(\prmtr)-\fB(\prmtr)\right\|\|\hat \inp(\lapVar_j)\|.
    \end{aligned}
\end{align*}
Then, dividing the residual norm by the denominator of \eqref{eqn:err:red:par:CIM} and using the definitions \eqref{eqn:res:ini:par:CIM}-\eqref{eqn:res:inp:par:CIM}, we end up with
\begin{align*}
    \begin{aligned}
        \frac{\|\res(\lapVar_j,\prmtr)\|}{\|\inSolSub\|+\max_{j=1,\ldots,N-1}\|\hat\inp(\lapVar_j)\|}\;\le\; \|\res_{\inSolSub}(\lapVar_j,\prmtr)\|+ \|\res_{\inp}(\lapVar_j,\prmtr)\|,
    \end{aligned}
\end{align*}
which concludes the proof after suitable substitutions in the derived inequalities.
\end{proof}
\begin{remark}
Similar techniques for estimating errors in parametric $\LCS$ have been already developed in, e.g., \cite{FenAB17,ZhangFB15,FengB21}. These resulting error bounds have then been used in a greedy procedure to construct effective \ROMs as we do here. Our primary contribution in this context is to integrate this approach with the \CIM formulation to derive new error bounds specific to the needs of the \CIM formulation and \CIM approximation~\eqref{eqn:stx:par:CIM}. This integration enables a direct correlation between frequency domain residuals and output errors within the time interval of interest. Moreover, note that $\Delta(\prmtr)$ does not depend on the initial solution $\inSolSub$ or the input function $\inp$. Consequently, when $\Delta(\tilde \prmtr)\le\varepsilon$ for a particular $\tilde \prmtr$, the reduction error \eqref{eqn:err:red:par:CIM}, evaluated at $\tilde\prmtr$, is smaller than $\varepsilon$ for all $\inSolSub\in\R^{r_{\calF}}$ and $\inp\in\calU$. This property is crucial for the effectiveness and accuracy of the \CIM-\MOR. For the \emph{non-parametric \LCS case}, a bound for the $\mathcal{L}_2$ norm of the output error for all admissible initial conditions can be found in~\cite{BeaGM17,HeiRA11,SchV23}.
\end{remark}

Based on the upper bound established in~\Cref{lemma:err:est}, the projection matrices $\fV_j$ and $\fW_j$, for the indices $j=1,\ldots,N$, are generated by \Cref{alg:1}. It should be noted that this methodology represents a non-linear \MOR strategy (see \cite{HesPU26} for a recent survey on the topic) since the projection matrices are piecewise constant within the frequency domain. This characteristic comes from the use of a singular quadrature rule in all parameters, initial solutions, and input functions. This aspect is critically significant for two primary reasons: it considerably diminishes the dimensionality of the linear systems that must be solved for the \ROM evaluation (see also \cite[Sec.~4]{GugM23}), and it facilitates the implementation of parallel processing during the construction of $\fV_j$ and $\fW_j$, given their independence from projection matrices at other quadrature points. We also point out that a global variant of the \CIM-\MOR can be defined where one can construct projection matrices $\fV$ and $\fW$ constant across all quadrature points, as demonstrated in \cite[Alg. $3.1$]{GugM23}. A comparison of the local and global \MOR in the context of \CIM is detailed in \cite[Sec.~4]{GugM23}, with findings that are fairly intuitive: the local approach results in significantly more precise \ROMs for a specified \ROM size due to the localized nature of the basis, contributing to quicker offline phases since the greedy algorithm achieves its exit tolerance more rapidly as a result of enhanced accuracy. However, the disadvantage of this technique lies in the necessity of maintaining distinct reduced matrices for each quadrature point, leading to storage costs that are roughly $\calO(N)$ higher than its global counterpart.

\begin{algorithm}[t]
       \caption{The \CIM-\MOR for parametric \LCS: offline phase}\label{alg:1}
		\hspace*{\algorithmicindent} \textbf{Input:} a compact domain $\prmtrSet\subseteq\R^{\prmtrDim}$, the map $\lapVar: \intvar\rightarrow \lapVar(\intvar)$ for the profile $\Gamma$, the number of quadrature points $N$, target global accuracy $\varepsilon>0$, local accuracy $\varepsilon_j=\varepsilon/(N-1)$, the matrices in \eqref{eqn:CS:par:mat}, and $\inMat$.
        
		\hspace*{\algorithmicindent} \textbf{Output:} The projection spaces $\fV_{j}$ and $\fW_{j}$, for $j=1,\ldots,N-1$, such that $\max_{\prmtr\in\prmtrSet}\Delta(\prmtr)\le \varepsilon$.
        
	\begin{algorithmic}[1]
			\FOR{$j=1:N-1$}
            \STATE Set $J \leftarrow 1$, choose an initial point $\prmtr_1 \in \prmtrSet$. 
	\STATE
    Compute $\fV_j(J)\leftarrow\mathrm{orth}\left(\begin{bmatrix}\left(\lapVar_j\fE(\prmtr_1)-\fA(\prmtr_1)\right)^{-1}\inMat(\prmtr_1) &\left(\lapVar_j\fE(\prmtr_1)-\fA(\prmtr_1)\right)^{-1}\fB(\prmtr_1)\end{bmatrix}\right)$.\label{alg:1:line3}
    \STATE
    Compute $\fW_j(J)\leftarrow\mathrm{orth}\left(\begin{bmatrix}
        \left(\lapVar_j\fE(\prmtr_1)-\fA(\prmtr_1)\right)^{-*}\fC(\prmtr_1)^{*}\end{bmatrix}\right)$.\label{alg:1:line4}
        
            \STATE $\Delta_j(\prmtr)\leftarrow w_j\|(\lapVar_j\fE(\prmtr)-\fA(\prmtr))^{-1}\|\|\res_{\hat \fp}(\lapVar_j,\prmtr)\| \left(\|\res_{\inSolSub}(\lapVar_j,\prmtr)\|+ \|\res_{\inp}(\lapVar_j,\prmtr)\|\right)$.
			\WHILE{$\max_{\prmtr\in\prmtrSet}\Delta_j(\prmtr)>\varepsilon_j$}
			\STATE Choose $\prmtr_{J+1} \leftarrow \arg \max_{\prmtr \in \prmtrSet   } \Delta_j(\prmtr) $. 
            \STATE  $\fV_j(J+1)\leftarrow\mathrm{orth}\left(\begin{bmatrix}\fV_j(J) &\left(\lapVar_j\fE(\prmtr_{J+1})-\fA(\prmtr_{J+1})\right)^{-1}\begin{bmatrix}\inMat(\prmtr_{J+1}) &\fB(\prmtr_{J+1}) \end{bmatrix}\end{bmatrix}\right)$.\label{alg:1:line8}
            \STATE  $\fW_j(J+1)\leftarrow\mathrm{orth}\left(\begin{bmatrix}\fW_j(J) &\left(\lapVar_j\fE(\prmtr_{J+1})-\fA(\prmtr_{J+1})\right)^{-*}\fC(\prmtr_{J+1})^{*}\end{bmatrix}\right)$.\label{alg:1:line9}
            \STATE Set $J\leftarrow J+1$
			\ENDWHILE
			\ENDFOR
		\end{algorithmic}
	\end{algorithm}
  
    \begin{remark}\label{rmk:dim:Alg1}
        Typically, the number of vectors incorporated in the subspaces $\fV_j$ and $\fW_j$ may vary with each iteration. To guarantee the invertibility of $\lapVar_j\fE_{\rr,j}(\prmtr)-\fA_{\rr,j}(\prmtr)$, it is necessary that $\fV_j$ and $\fW_j$ possess equal dimensions. To achieve this, in each iteration $J$, we introduce randomly generated vectors into the columns of the smaller subspace, selected between the column spaces of $\fV_j(J+1)$ and $\fW_j(J+1)$.
    \end{remark}

For the reader who is familiar with frequency-domain interpolatory methods for parametric \MOR~\cite{BauBBG11,BenGW15,AntBG20}, the interpolatory structure of the subspaces $\fV_j$ and  $\fW_j$ resulting from~\Cref{alg:1} is clear. Even though producing interpolatory \ROMs has not been the main goal here, the output approximation $\out_N(t,\prmtr)$ of the \CIM formulation  in~\eqref{eqn:out:par:CIM} via the inverse Laplace transform together with the error estimate in~\Cref{lemma:err:est} and the resulting greedy procedure naturally leads to interpolatory subspaces via frequency domain sampling. 

Due to the inhomogeneous initial condition, as highlighted in~\eqref{eqn:Bro:int2} there are two rational functions in play here as opposed to a single transfer function to interpolate. One may still reformulate the resulting interpolatory conditions as an interpolation of a single rational function by expanding the input matrix $\fB$ with the initial condition subspace $\fE(\prmtr)\fF$ and then using the results of~\cite{BauBBG11} on this augmented system to prove the resulting interpolation conditions in the next result. However, since the \CIM~formulation we consider here focuses on the output in a specific time interval only, we instead present the result as interpolation of the output $\out_N(t,\prmtr)$ in the frequency domain for any suitable initial condition.  Towards this goal, let us denote by $\calJ_j\vcentcolon=\cup_{J=1}^{J_{j,\max}}\prmtr_J$ the set of parameters chosen by \Cref{alg:1} for the generation of subspaces at the quadrature point $\lapVar_j$. The subsequent result establishes that $\hat \out_{\rr,N}(\lapVar_j,\prmtr)$ is an Hermite interpolant to $\out_{N}(\lapVar_j,\prmtr)$ for every $\prmtr\in\calJ_j$. In other words, the Laplace transform of the output $\out_{N}(\lapVar_j,\prmtr)$ at the quadrature point $\lapVar_j$ is Hermite-interpolated by $\hat \out_{\rr,N}(\lapVar_j,\prmtr)$ for all $\prmtr\in\calJ_j$ as anticipated by the direct connection to the frequency-domain interpolatory projection methods.

\begin{lemma}[Hermite interpolation property of the \CIM-\MOR]\label{lem:int:cond}
Consider the Laplace transform of the \FOM output at the quadrature point $\lapVar_j$, i.e., 
\begin{equation}\label{eqn:out:FOM:lap}
     \hat \out_{N}(\lapVar_j,\prmtr)\;=\;\fD(\prmtr)\hat \inp(\lapVar_j)+\fC(\prmtr)\hat\stx(\lapVar_j,\prmtr),
\end{equation}
and the correspending \ROM version
\begin{equation}\label{eqn:out:ROM:lap2}
     \hat \out_{\rr,N}(\lapVar_j,\prmtr)\;=\;\fD(\prmtr)\hat \inp(\lapVar_j)+\fC_{\rr,j}(\prmtr)\hat\stx_{\rr}(\lapVar_j,\prmtr),
\end{equation}
where the \ROM is constructed toward \Cref{alg:1} after selecting, for each $\lapVar_j$,
$J_{j,\max}$ parameters that form the set $\calJ_j=\cup_{J=1}^{J_{j,\max}}\prmtr_J$. We have that, for all $j\in\{1,\ldots,N-1\}$, it holds 
    \begin{align}
        \hat \out_{N}(\lapVar_j,\prmtr)\;=&\;\hat \out_{\rr,N}(\lapVar_j,\prmtr),\label{eqn:int:point}\\
       \nabla \hat \out_{N}(\lapVar_j,\prmtr)\;=&\;\nabla \hat \out_{\rr,N}(\lapVar_j,\prmtr),\label{eqn:int:deriv}
    \end{align}
 for all $\prmtr\in\calJ_j$.       
\end{lemma}
\begin{proof}
By \eqref{eqn:out:FOM:lap}, \eqref{eqn:out:ROM:lap2}, \eqref{eqn:rstate:par:CIM}, and \eqref{eqn:out:par:CIM:red} we have that, for all $\prmtr\in\calJ_j$, it holds
\begin{equation*}
    \|\hat \out_{N}(\lapVar_j,\prmtr)-\hat \out_{\rr,N}(\lapVar_j,\prmtr)\|\le\|\fC(\prmtr)\|\|\hat\stx(\lapVar_j,\prmtr)-\fV_j\hat\stx_{\rr}(\lapVar_j,\prmtr)\|.
\end{equation*}
 Using the extension of Cea's lemma for Petrov-Galerkin projections and the equivalence of norms in finite dimension, we get
  \begin{equation}\label{eqn:cea:lemma}
      \|\hat\stx(\lapVar_j,\prmtr)-\fV_j\hat\stx_{\rr}(\lapVar_j,\prmtr)\|\;\le\;p(\lapVar_j,\prmtr)\inf_{\fw\in\fV_j}\|\hat\stx(\lapVar_j,\prmtr)-\fw\|\;=\;0
  \end{equation}
    for $0<p(\lapVar_j,\prmtr)<\infty$ for all $\lapVar_j$ with $j=1,\ldots,N-1$ and $\prmtr\in\prmtrSet$ due to the first assumption discussed in \Cref{sec3}. The equality with zero arises from recognizing that, according to \eqref{eqn:stx:par:CIM} and given that $\stx_0=\inMat\inSolSub$, we have:
\begin{equation*}
    \hat\stx(\lapVar_j,\prmtr)\in \range\left(\left(\lapVar(\intvar)\fE(\prmtr)-\fA(\prmtr)\right)^{-1}\begin{bmatrix}
        \inMat(\prmtr)&\fB(\prmtr)
    \end{bmatrix}\right).
\end{equation*}
Additionally, by the construction of $\fV_j$, specifically in lines \ref{alg:1:line3} and \ref{alg:1:line8}, it is evident that $\hat\stx(\lapVar_j,\prmtr)\in \fV_j$ for all $\prmtr\in\calJ_j$. Thus, we deduce that \eqref{eqn:cea:lemma} and, consequently, \eqref{eqn:int:fun} are valid.

We prove \eqref{eqn:int:deriv} by showing that the partial derivative with respect to the parameter component $\mu_i$ is identical for all $i=1,\ldots,\prmtrDim$ and all $\prmtr\in\calJ_j$. For simplicity, we consider $\fD(\prmtr)=\zeroVec$ and define $\fM_j(\prmtr)=\lapVar_j\fE(\prmtr)-\fA(\prmtr)$ and by $\fM_{\rr,j}(\prmtr)$ its reduced counterpart according to the \CIM-\MOR approach. We thus have
\begin{align*}
    \begin{aligned}
        \frac{\partial}{\partial \mu_i}\left( \hat \out_{\rr,N}(\lapVar_j,\prmtr)\right)\;=&\;\underbrace{\frac{\partial \fC_{\rr,j}(\prmtr)}{\partial \mu_i}(\fM_{\rr,j}(\prmtr))^{-1}\left(\fX_{\stateDimRed,j}(\prmtr)\inSolSub+\fB_{\rr,j}(\prmtr)\hat\inp(\lapVar_j)\right)}_{=\vcentcolon\;\fS_1(\lapVar_j,\prmtr)}\\
        +&\;\underbrace{\fC_{\rr,j}(\prmtr)\frac{\partial }{\partial \mu_i}\left((\fM_{\rr,j}(\prmtr))^{-1}\right)\left(\fX_{\stateDimRed,j}(\prmtr)\inSolSub+\fB_{\rr,j}(\prmtr)\hat\inp(\lapVar_j)\right)}_{=\vcentcolon\;\fS_2(\lapVar_j,\prmtr)}\\
        +&\;\underbrace{\fC_{\rr,j}(\prmtr)(\fM_{\rr,j}(\prmtr))^{-1}\frac{\partial }{\partial \mu_i}\left(\fX_{\stateDimRed,j}(\prmtr)\inSolSub+\fB_{\rr,j}(\prmtr)\hat\inp(\lapVar_j)\right)}_{=\vcentcolon\;\fS_3(\lapVar_j,\prmtr)}.
    \end{aligned}
\end{align*}
For $\boldsymbol{\mathrm{A}}_j(\prmtr)$, by \eqref{eqn:rstate:par:CIM} and using the relation \eqref{eqn:cea:lemma} that holds for all $\prmtr\in\calJ_j$, we can write
\begin{align*}
    \begin{aligned}
 \fS_1(\lapVar_j,\prmtr)\;=&\;\frac{\partial \fC(\prmtr)}{\partial \mu_i}\fV_j\hat{\stx}_{\rr}(\lapVar_j,\prmtr)\;=\;\frac{\partial \fC(\prmtr)}{\partial \mu_i}\hat{\stx}(\lapVar_j,\prmtr),
    \end{aligned}
\end{align*}
where we also employed the affine parameter dependence of $\fC(\prmtr)$ that allows the relation
\begin{equation*}
    \frac{\partial \fC_{\rr,j}(\prmtr)}{\partial \mu_i}\;=\;\frac{\partial \fC(\prmtr)}{\partial \mu_i}\fV_j.
\end{equation*}
Analogously, for $\fS_3(\lapVar_j,\prmtr)$ by \eqref{eqn:CIM:red:adj} it holds for all $\prmtr\in\calJ_j$ that
\begin{align*}
    \begin{aligned}
\fS_3(\lapVar_j,\prmtr)=&\hat\fp_{\rr}(\lapVar_j,\prmtr)^{*}\fW_{j}^{*}\frac{\partial}{\partial \mu_i}\left(\inMat(\prmtr)\inSolSub+\fB(\prmtr)\hat\inp(\lapVar_j)\right)=\hat\fp(\lapVar_j,\prmtr)^{*}\frac{\partial}{\partial \mu_i}\left(\inMat(\prmtr)\inSolSub+\fB(\prmtr)\hat\inp(\lapVar_j)\right),
    \end{aligned}
\end{align*}
where the equality $\hat\fp(\lapVar_j,\prmtr)=\fW_j\hat\fp_{\rr}(\lapVar_j,\prmtr)$ for every $\prmtr\in\calJ_j$ is derived analogously to the proof for $\hat{\stx}(\lapVar_j,\prmtr)$ and $\fV_j\hat{\stx}_{\rr}(\lapVar_j,\prmtr)$. This is achieved by noticing that $\hat\fp(\lapVar_j,\prmtr)$ is in the range of $\fM_j(\prmtr)^{-*}\fC(\prmtr)^{*}$, considering the construction process of $\fW_j$ described in lines \ref{alg:1:line4} and \ref{alg:1:line8}, and ultimately using Cea's lemma. Deriving with respect to $\mu_i$ the relation $\fM_{\rr,j}(\prmtr)\fM_{\rr,j}(\prmtr)^{-1}=\fI_{\stateDimRed_j}$, we get 
\begin{align}\label{eqn:der:mat:inv}
\begin{aligned}
    \frac{\partial }{\partial \mu_i}\left(\fM_{\rr,j}(\prmtr)^{-1}\right)\;=&\;(\fM_{\rr,j}(\prmtr))^{-1} \frac{\partial \fM_{\rr,j}(\prmtr)}{\partial \mu_i}(\fM_{\rr,j}(\prmtr))^{-1}\\
    =&\;(\fM_{\rr,j}(\prmtr))^{-1}\fW_j^{*} \frac{\partial \fM_{j}(\prmtr)}{\partial \mu_i}\fV_j(\fM_{\rr,j}(\prmtr))^{-1},
    \end{aligned}
\end{align}
which allows us to write, in the same way as done for $\fS_1(\lapVar_j,\prmtr)$ and $\fS_3(\lapVar_j,\prmtr)$, that
\begin{align*}
    \begin{aligned}
      \fS_2(\lapVar_j,\prmtr)\;=&\; \fC_{\rr,j}(\prmtr)(\fM_{\rr,j}(\prmtr))^{-1}\fW_j^{*} \frac{\partial \fM_{j}(\prmtr)}{\partial \mu_i}\fV_j(\fM_{\rr,j}(\prmtr))^{-1}\left(\fX_{\stateDimRed,j}(\prmtr)\inSolSub+\fB_{\rr,j}(\prmtr)\hat\inp(\lapVar_j)\right) \\
      =&\;\hat\fp_{\rr}(\lapVar_j,\prmtr)^{*}\fW_{j}^{*}\frac{\partial \fM_{j}(\prmtr)}{\partial \mu_i}\fV_j\hat\stx_{\rr}(\lapVar_j,\prmtr)\;=\;\hat\fp(\lapVar_j,\prmtr)^{*}\frac{\partial \fM_{j}(\prmtr)}{\partial \mu_i}\hat\stx(\lapVar_j,\prmtr);
    \end{aligned}
\end{align*}
for all $\prmtr\in\calJ_j$ due to the interpolation conditions for the reduced state and adjoint state. The derivation of \eqref{eqn:out:FOM:lap} with respect $\mu_i$, for $\fD(\prmtr)=\zeroVec$, $\prmtr\in\calJ_j$, and any $i\in\{1,\ldots,\prmtrDim\}$ gives
\begin{align*}
    \begin{aligned}
        \frac{\partial}{\partial \mu_i}\left( \hat \out_{N}(\lapVar_j,\prmtr)\right)\;=&\;{\frac{\partial \fC(\prmtr)}{\partial \mu_i}\fM_{j}(\prmtr)^{-1}\left(\inMat(\prmtr)\inSolSub+\fB(\prmtr)\hat\inp(\lapVar_j)\right)}\\
        +&\;{\fC(\prmtr)\frac{\partial }{\partial \mu_i}\left(\fM_{j}(\prmtr)^{-1}\right)\left(\inMat(\prmtr)\inSolSub+\fB(\prmtr)\hat\inp(\lapVar_j)\right)}\\
        +&\;{\fC(\prmtr)\fM_{j}(\prmtr)^{-1}\frac{\partial }{\partial \mu_i}\left(\inMat(\prmtr)\inSolSub+\fB(\prmtr)\hat\inp(\lapVar_j)\right)}\\
        =&\;\frac{\partial \fC(\prmtr)}{\partial \mu_i}\hat \stx(\lapVar_j,\prmtr)+\hat\fp(\lapVar_j,\prmtr)^{*}\frac{\partial}{\partial \mu_i}\left(\inMat(\prmtr)\inSolSub+\fB(\prmtr)\hat\inp(\lapVar_j)\right)\\
        +&\;\hat\fp(\lapVar_j,\prmtr)^{*}\frac{\partial \fM_{j}(\prmtr)}{\partial \mu_i}\hat\stx(\lapVar_j,\prmtr)\\
=\;&\fS_1(\lapVar_j,\prmtr)+\fS_3(\lapVar_j,\prmtr)+\fS_2(\lapVar_j,\prmtr)\;=\;\frac{\partial}{\partial \mu_i}\left( \hat \out_{\rr,N}(\lapVar_j,\prmtr)\right),
    \end{aligned}
\end{align*}
where we used \eqref{eqn:stx:par:CIM}, \eqref{eqn:adj:par:CIM}, and the fact that for the derivative of $\fM_{j}(\prmtr)^{-1}$ it holds an analogous relation to \eqref{eqn:der:mat:inv}. This concludes the proof.
\end{proof}
\begin{remark}
As discussed above, interpolation results for the frequency domain Hermite interpolation conditions of parametric \LCS already exists in \cite[Thm.~4.1-4.2]{BauBBG11}. The assumptions necessary in \cite{BauBBG11} for these conditions are indeed met within the \CIM-\MOR framework, thereby automatically validating \eqref{eqn:int:point}-\eqref{eqn:int:deriv}. However, we stress that our approach is intrinsically rooted in the \CIM methodology, which relies on an automatic selection of interpolation parameters via a greedy algorithm grounded in an error estimate for the \CIM approximation. This process inherently ensures the fulfillment of the Hermite interpolation conditions together with the error minimization~as achived in~\cite{FenAB17,ZhangFB15,FengB21} for different error criteria. 

Our formulation also accommodates non-zero initial conditions, and the statement and proof differ from those in \cite[Thm.~4.1--4.2]{BauBBG11}. Therefore, for the sake of clarity and completeness, we still included the proof here.
\end{remark}

\begin{remark}
    In \cite[Thm.~$4.3$]{BauBBG11}, conditions are provided for the interpolation of the Hessian of the transfer function with respect to the dependence on the parameters. Within our \CIM-\MOR framework, these interpolation conditions would also be applicable under identical assumptions. This approach could be used to maintain the same dimension for $\fV_j$ and $\fW_j$ throughout iterations $J$. Specifically, rather than adhering to \Cref{rmk:dim:Alg1}, which is based on the addition of randomly generated vectors, one could augment the subspaces with vectors that enforce directional interpolations of the Hessian.
\end{remark}

    \begin{algorithm}[t]
       \caption{The \CIM-\MOR for parametric \LCS: online phase}\label{alg:2}
		\hspace*{\algorithmicindent} \textbf{Input:} A triplet $(\prmtr,\inp, \inSolSub)$ with $\prmtr\in\prmtrSet$ and $\inp_0\in\calU$, $t\in[T,\Lambda T]$, the map $\lapVar: \intvar\rightarrow \lapVar(\intvar)$ for the profile $\Gamma$, the number of quadrature points $N$, pre-assembled matrices using $\fV_j$, $\fW_j$, $\inMat$, and the assumption \eqref{eqn:CS:par:mat}.
        
		\hspace*{\algorithmicindent} \textbf{Output:} The reduced output $\out_{\mathrm{r},N}(t,\prmtr)$.
        
	\begin{algorithmic}[1]
			\FOR{$j=1:N-1$}
            \STATE Assemble $\fE_{\stateDimRed,j}(\prmtr)$, $\fA_{\stateDimRed,j}(\prmtr)$, $\fB_{\stateDimRed,j}(\prmtr)$, $\fC_{\stateDimRed,j}(\prmtr)$, and $\fX_{\stateDimRed,j}(\prmtr)$ according  \eqref{eqn:proj:mat} and \eqref{eqn:out:par:CIM:red}.
            \STATE Compute and store $\hat\stx_{\mathrm{r}}(\lapVar_j,\prmtr)=(\lapVar_j\fE_{\rr,j}(\prmtr)-\fA_{\rr,j}(\prmtr))^{-1}\left(\fX_{\stateDimRed,j}(\prmtr)\inSolSub+\fB_{\rr,j}(\prmtr)\hat\inp(\lapVar_j)\right)$.
			\ENDFOR
            \STATE Evaluate $\out_{\mathrm{r},N}(t,\prmtr)$ via \eqref{eqn:out:par:CIM:red}.
		\end{algorithmic}
	\end{algorithm}

   \Cref{alg:2} summarizes the online phase of the \CIM-\MOR procedure. Given a parameter $\prmtr\in\prmtrSet$, an input $\inp\in\calU$, and an initial condition $\inSolSub$, the dominant computational cost associated with the evaluation of the reduced output is the solution of a small number of linear systems of dimension $\stateDimRed\ll\stateDim$. Moreover, the same considerations discussed in \Cref{sec2:comp-cost} apply in the case of a fixed parameter $\prmtr$. Specifically, for a given $\prmtr$, one may preassemble the reduced counterparts of the matrices in \eqref{eqn:small:size:mat} and subsequently evaluate the output for all $\inp\in\calU$ and $\stx_0\in\calF$ without the need to solve any additional linear systems.

\subsection{About the computational complexity of \texorpdfstring{$\Delta(\prmtr)$}{EMPTY}}\label{sec:3.3}
As stated above, an effective error estimator must be computable with efficiency for every $\prmtr\in\prmtrSet$. This implies that the total floating-point operations required for its assessment should not be dependent on $\stateDim$.
The parameter dependence of $\Delta (\prmtr )$ appears in both the residual norms and the norm of $\|(\lapVar_j\fE(\prmtr)-\fA(\prmtr))^{-1}\|$. The former can be calculated efficiently and with stability for each $\prmtr$, as described in \cite{BuhEOR14}. The latter involves determining the smallest singular value of $\lapVar_j\fE(\prmtr)-\fA(\prmtr)$, with its computational expense proportional to $\stateDim$. Over the past ten years, subspace approaches have been introduced to tackle this issue both accurately and efficiently. Initially, \cite{SirK16} enhanced the traditional \textit{successive constraint minimization} (\SCM) technique \cite{HuyRSP07} to estimate the smallest eigenvalue in parametric Hermitian matrices. In short, given a parametric matrix in affine form $\fA(\prmtr)=\sum_{j=1}^{Q_{\fA}}\alpha_j(\prmtr)\fA_j$, where each $\fA_i$ is Hermitian for $j=1,\ldots,Q_{\fA}$, the subspace method seeks to approximate the smallest eigenvalue of $\fA(\prmtr)$, denoted by $\lambda_{\min}(\fA(\prmtr))$. This is achieved by substituting $\lambda_{\min}(\fA(\prmtr))$ with the smallest eigenvalue of a possibly much smaller matrix $\fA_{\fV}(\prmtr)=\sum_{j=1}^{Q_{\fA}}\alpha_j(\prmtr)\fV^{*}\fA_j\fV$. Here, $\fV$ is a matrix that contains a limited number of orthonormal columns. Essentially, this technique represents a projection-based \MOR for the eigenvalue problem. More recently, \cite{ManMG24} provided a deep theoretical analysis of the approach by \cite{SirK16}, discussing how to extend it to a continuous parametric domain and presenting a novel subspace framework to approximate the smallest singular value. We employ this methodology to approximate $\|(\lapVar_j\fE(\prmtr)-\fA(\prmtr))^{-1}\|$ in this work. However, we remark that despite this 
subspace framework for singular values, it is shown to work very well, both in terms of efficency and accuracy (see \cite[Sec.~6.2]{ManMG24}), the method proposed does not always guarantee that the approximation error is below a pre-defined target accuracy; see \cite[Rmk.~8]{ManMG24}. The recent work~\cite{BalEG26} has developed the Keldysh decomposition for matrix-valued functions that depend analytically on a parameter and used contour integration techniques develop an algorithm for solving parametric nonlinear eigenvalue problems. Since~\cite{BalEG26} uses the frequency domain samples in a similar form which we already sample here, it might be another avenue to incorporate into this setting to reduce the computational cost.

To conclude, we emphasize that after constructing the \ROM using the \CIM-\MOR methodology, one can precompute the matrices, as outlined in \Cref{sec2:comp-cost}, for any parameter $\prmtr\in\prmtrSet$. This allows for rapid evaluations across all valid inputs and initial solutions.

\section{Numerical results}\label{sec4}
We demonstrate the accuracy and efficiency of the \CIM-\MOR technique applied to linear control systems using several benchmark examples. We then compare it with state-of-the-art methodologies for \MOR of control systems. \Cref{subsection:non-parametric} focuses on the non-parametric scenario, whereas \Cref{subsection:parametric} addresses parametric control systems. In every test example presented, leveraging the symmetry of the integrand function \eqref{eqn:int:fun} with respect to the real axis allows for reducing the utilized number of quadrature points by half. Consequently, when the notation $N$ refers to the total number of quadrature points appears in the plots of \Cref{subsection:non-parametric}, it should be understood that it has already been halved. In other words, $N$ here represents the total number of linear systems to be solved.

All calculations were performed with \matlab 2024b on a MacBook Pro with an Apple M2 Pro processor and 16GB of RAM. 

\vspace{0.2cm}
\noindent\fbox{%
	\parbox{0.98\textwidth}{%
		The code and data used to generate the subsequent results are accessible via
		\begin{center}
			\url{https://doi.org/10.5281/zenodo.21298216}
		\end{center}
		under MIT Common License.
	}%
}
\subsection{Non-parametric problems}\label{subsection:non-parametric}
\Cref{tab:T0} list the benchmark problems considered to test the \CIM-\LCS.

\begin{table}[h!]
\centering
\renewcommand{\arraystretch}{1.5}
\begin{tabular}{lrcc}
\toprule
Example&$\stateDim$ & $\inpDim$ & $\outDim$ \\
\midrule
International space station & $1412$ & $3$ & $3$  \\
Thermal block & $7488$& 1& 4 \\
Butterworth filter & 300 & 1 & 1 \\
\bottomrule
\end{tabular}
\caption{A list of examples with their dimensions ($\stateDim$), the number of inputs
($\inpDim$), and outputs ($\outDim$). These examples are taken from 	\url{  https://modelreduction.org/morwiki/Main_Page}.}\label{tab:T0}
\end{table}

\subsubsection{International space station model} 
We begin by examining the international space station model, specifically the one outlined in \cite{morGugAB01}, with $\fE=\fI_{\stateDim}$ and $\stateDim=1412$, and both the input and output dimensions are $\inpDim=3$ and $\outDim=3$, respectively. The feedthrough matrix $\fD$ is zero. We aim to approximate the solution for all $t\in[T,\Lambda T]$ with $T=50$, $\Lambda=2$, and the initial conditions in the span of the columns of the randomly generated matrix $\inMat=\orth(\randn(\stateDim,\stateDimRed_{\calF}))$, for $\stateDimRed_{\calF}=3$. After determining the curves $\Gamma_{left}$, $\Gamma_{right}$, and the integration profile $\Gamma$ following the procedure detailed in \Cref{subsection:2.1}, see \Cref{figISSc}, we test the \CIM-\LCS approach with the following input functions 
$\inp(t)$  and initial conditions $\stx_0 = \inMat \inSolSub$:
\begin{subequations}
    \begin{align}
        \inp(t)\;=&\;\left[\sin(t),t,\cos(t)\right]^{\T},\quad\quad\quad\quad\quad\quad\quad\quad\quad\inSolSub\;=\;\begin{bmatrix}
            1& 1&1
        \end{bmatrix}^\T,\label{eqn:ISS:inp1}\\
\inp(t)\;=&\;\left[t^2\ee^{-\tfrac{t}{50}},\sinh\left(\tfrac{t}{100}\right),\cosh\left(\tfrac{t}{50}\right)\right]^{\T},\quad\quad\inSolSub\;=\;\rand(\stateDimRed_{\calF},1).\label{eqn:ISS:inp2}
    \end{align}
\end{subequations}
\begin{figure}[t]
	\centering{
	\begin{subfigure}[t]{0.32\textwidth}
	    \input{img/ISS/Example_ISS_1.tex}
        \subcaption{Decay of $\calE(N)$ (see \eqref{eqn:err:def}) for different $\tol$; $\inp$ and $\inSolSub$ as in \eqref{eqn:ISS:inp1}.}\label{figISSa}
	\end{subfigure}
    \hfill
    \begin{subfigure}[t]{0.32\textwidth}
	    \input{img/ISS/Example_ISS_2.tex}
         \subcaption{Decay of $\calE(N)$ (see \eqref{eqn:err:def}) for different $\tol$; $\inp$ and $\inSolSub$ as in \eqref{eqn:ISS:inp2}.}\label{figISSb}
	\end{subfigure}
    \hfill
    \begin{subfigure}[t]{0.32\textwidth}
	    \input{img/ISS/Example_ISS_3.tex}
         \subcaption{The spectrum of $\fA$, $\Gamma_{left}$, $\Gamma_{right}$, and $\Gamma$ for $\tol=10^{-4}$.}\label{figISSc}
	\end{subfigure}
 }
	\caption{International space station benchmark for $T=50$, $\Lambda=2$, and $\stateDimRed_{\calF}=3$.}
	\label{fig1}%
\end{figure}
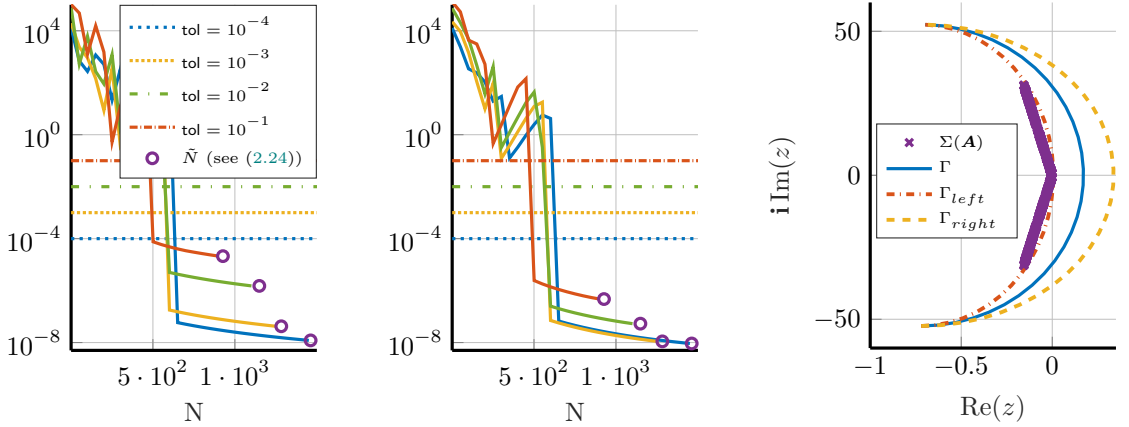%
Results are displayed in \Cref{figISSa} and \Cref{figISSb}. The \CIM-\LCS always achieves the desired accuracy $\tol$ for both the considered combinations of input functions and initial conditions; moreover, the number of quadrature nodes to reach the target accuracy, i.e., $\tilde N$, is successfully predicted. 

This example involves solving numerous linear systems of approximately $\calO(10^3)$, which may be considered extensive for a system with a state dimension of $1412$. It is crucial to highlight that the number of quadrature points utilized should be evaluated against the number of time steps that standard Runge-Kutta or multi-step methods would require for temporal integration at the specified accuracy $\tol$, rather than being compared to the system's dimension $\stateDim$. As demonstrated in \Cref{subsec:4.1 MOR comparison}, even though $N$ is in the thousands, the capacity to precompute the primary computational load only once considerably accelerates the online phase of the \CIM-\LCS, as opposed to any time-stepping integrator used on small systems obtained through a \MOR technique. The reason for the large value of $N$ in this example is linked to the spectral properties and the resolvent norm associated with this benchmark. As evident in \Cref{figISSc}, several eigenvalues possess a significantly large imaginary component, hindering the efficient transformation of the vertical line into the contour $\Gamma$. Nonetheless, advancements can be achieved in handling such issues. For instance, one can leverage the composition of quadrature rules \cite{GuoL25} for sectorial operators to decrease the number of nodes required or utilize the differentiability of the input function and initial solution, when accessible, to enhance the convergence rate of the quadrature rules; see \cite{HorG24}.

\subsubsection{Non-parametric thermal block} \label{sec:TB:ex}
We examine the non-parametric form of the semi-discretized heat transfer model detailed in \cite{morwiki_thermalblock}. Specifically, the model involves sparse matrices, resulting in a state-space dimension of $\stateDim=7488$, with an input dimension of $\inpDim=1$, and $\outDim=4$ outputs. It is important to note that, for this problem, the matrix $\fE$ is not equivalent to the identity matrix. We look for an approximation of $\out(T)$ for $T=100$. \Cref{figTBc} displays the contour map for $\tol=10^{-6}$ and some elements of the generalized spectrum of $(\fA,\fE)$ which we observed to lie entirely on the negative real axis. 
We pick $\inMat=\orth(\randn(\stateDim,\stateDimRed_{\calF}))$ with  $\stateDimRed_{\calF}=10$. Then,
The input functions $\inp(t)$ and the initial conditions $\stx_0 = \inMat \inSolSub$ considered for testing this benchmark are
\begin{subequations}\label{eqn:TB:initial:data}
    \begin{align}
\inp(t) \;=&\; \sin\left(\frac{t}{2}\right),   & \inSolSub \;=&\; \rand(\stateDimRed_{\calF},1),  \label{eqn:TB:inp1}\\
\inp(t) \;=&\; t,         & \inSolSub \;=&\; \begin{bmatrix}
    1&\cdots &1
\end{bmatrix}^{\T}. \label{eqn:TB:inp2}
    \end{align}
\end{subequations}
The results are shown in \Cref{figTBa} and \Cref{figTBb}.
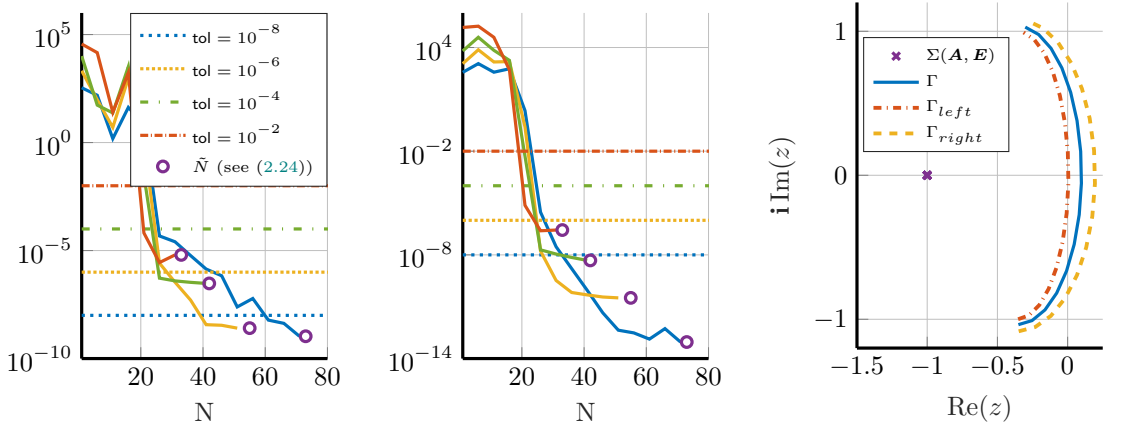
\begin{figure}[t]
	\centering{
	\begin{subfigure}[t]{0.32\textwidth}
	    \input{img/TB/Example_TB_1.tex}
        \subcaption{Decay of $\calE(N)$ (see \eqref{eqn:err:def}) for different $\tol$; $\inp$ and $\inSolSub$ as in \eqref{eqn:TB:inp1}.} \label{figTBa}
	\end{subfigure}
    \hfill
    \begin{subfigure}[t]{0.32\textwidth}
	    \input{img/TB/Example_TB_2.tex}
        \subcaption{Decay of $\calE(N)$ (see \eqref{eqn:err:def}) for different $\tol$; $\inp$ and $\inSolSub$ as in \eqref{eqn:TB:inp2}.} \label{figTBb}
	\end{subfigure}
    \hfill
    \begin{subfigure}[t]{0.32\textwidth}
	    \input{img/TB/Example_TB_3.tex}
         \subcaption{The spectrum of $(\fA,
         \fE)$, $\Gamma_{left}$, $\Gamma_{right}$, and $\Gamma$ for $\tol=10^{-6}$.}
         \label{figTBc}
	\end{subfigure}
 }
	\caption{Thermal block benchmark for $T=100$, $\Lambda=1$, and $\stateDimRed_{\calF}=10$.}
	\label{fig2}%
\end{figure}%
The error linked to the estimated number of quadrature points $\tilde N$ consistently remains under the specified accuracy threshold $\tol$ in both cases. Furthermore, compared to the International Space Station benchmark, the number of quadrature points needed is approximately an order of magnitude smaller. This is largely attributed to the generalized spectrum of $(\fA,\fE)$ being confined entirely to the real axis, along with the resolvent norm's limitation in magnitude away from the generalized spectrum, thereby rendering the \CIM highly efficient.

\subsubsection{Butterworth filter} \label{subsec:BF}
The third benchmark test considered is the digital Butterworth filter via its state-space representation with $\stateDim=300$, $\inpDim=1$, $\outDim=1$, and $\fE=\fI_{\stateDim}$; see \cite{Lyons1996UnderstandingDS} for more details. We approximate the output $\out(t)$ for all $t\in[T,\Lambda T]$ with $T=1$ and $\Lambda=5$. The integration profile used for $\tol=10^{-8}$ is in \Cref{figISSc}. We test the \CIM-\LCS approach for the pairs of initial data
\begin{subequations}\label{eqn:BF:initial:data}
    \begin{align}
\inp(t) \;=&\; \cos(t),   & \stx_0 \;=&\; \begin{bmatrix}
    1&1&1&1 &1
\end{bmatrix}^{\T},  \label{eqn:BF:inp1}\\
\inp(t) \;=&\; \cosh{\left(\frac{t}{5}\right)},         & \stx_0 \;=&\; \randn(\stateDimRed_{\calF},1), \label{eqn:BF:inp2}
    \end{align}
\end{subequations}
\begin{figure}[t]
	\centering{
	\begin{subfigure}[t]{0.32\textwidth}
	    \input{img/BG/Example_BG_1.tex}
        \subcaption{Decay of $\calE(N)$ (see \eqref{eqn:err:def}) for different $\tol$; $\inp$ and $\inSolSub$ as in \eqref{eqn:BF:inp1}.}	\label{figBFa}
	\end{subfigure}
    \hfill
    \begin{subfigure}[t]{0.32\textwidth}
	    \input{img/BG/Example_BG_2.tex}
        \subcaption{Decay of $\calE(N)$ (see \eqref{eqn:err:def}) for different $\tol$; $\inp$ and $\inSolSub$ as in \eqref{eqn:BF:inp2}.}	\label{figBFb}
	\end{subfigure}
    \hfill
    \begin{subfigure}[t]{0.32\textwidth}
	    \input{img/BG/Example_BG_3.tex}
         \subcaption{$\Gamma$ used for $\tol=10^{-8}$ and spectrum of $\fA$.}
	\end{subfigure}	\label{figBFc}
 }
	\caption{Butterworth filter test problem for $T=1$, $\Lambda=5$, and $\stateDimRed_{\calF}=5$.}
	\label{figBF}%
\end{figure}
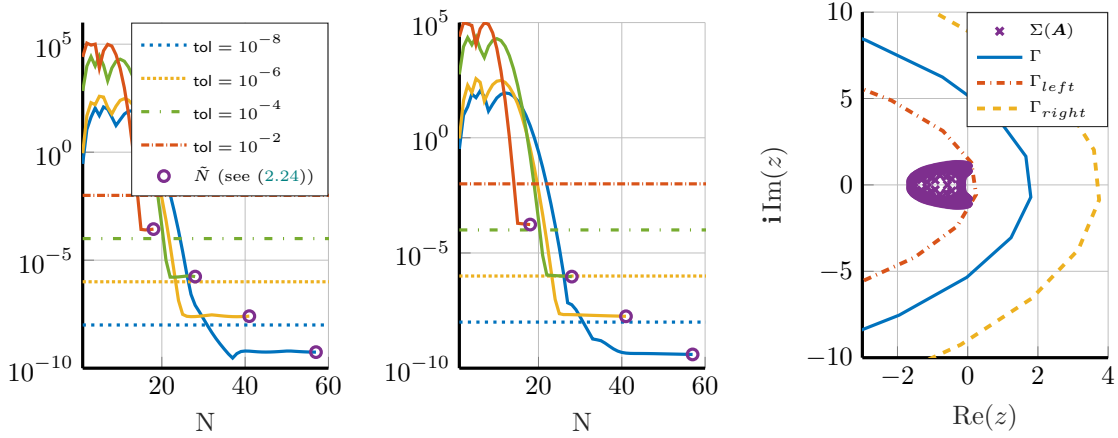%
being $\stateDimRed_{\calF}=5$ and $\inMat=\orth(\randn(\stateDim,\stateDimRed_{\calF}))$. The effectiveness of the methodology is demonstrated with \Cref{figBFa} and \Cref{figBFb}.

\subsubsection{Comparison with traditional \MOR for \LCS}\label{subsec:4.1 MOR comparison}
In this section, we compare the computational time associated with the \CIM-\LCS approach with that of the other two established projection-based reduction methods for control systems: balance truncation (\BT), specifically its time-limited version \cite{Kur18}, and the iterative rational Krylov algorithm (\IRKA), see \cite{gugercin_2008}. 

The computational expenses associated with the \CIM-\LCS are addressed in \Cref{sec2:comp-cost}. During the online phase, both \TLBT and \IRKA utilize time-stepping integration within the \ROM framework. In the offline phase, \TLBT necessitates computing the matrix exponential-vector product $\ee^{\fA T}\fB$ and $\ee^{\fA^{\T} T}\fC^{\T}$, followed by solving two Lyapunov equations. Conversely, \IRKA requires solving multiple linear systems of equations involving the matrix $\lapVar\fE-\fA$, until the reduced order model meets the $\calH_2$ optimality criteria. Adapting these projection-based techniques to general initial conditions is non-trivial, as their offline phase is dedicated to constructing projection spaces for one specific initial condition. A change in the initial solution demands revisiting parts of the offline phase to appropriately modify the projection spaces. For the matrix exponential-vector product, we used the function \expmv of \matlab, while to solve sparse Lyapunov equations, we employed \cite{SaaKB25}. For \IRKA, we used the function available in \cite{BreU21}.

\Cref{tab:T1} presents the computational times for two out of the three evaluated test examples with a target time of $T=100$, ensuring \eqref{eqn:err:def} remains below $10^{-5}$ for all the three methods considered. The offline stage of the \CIM-\LCS involved establishing the integration contour $\Gamma$ and assembling matrices represented by \eqref{eqn:small:size:mat}. The offline phase is primarily governed by the first component, which is evidently more resource-intensive compared to the offline phases of both \TLBT and \IRKA. Nevertheless, after $\Gamma$ has been established, the online stage of the \CIM-\LCS method is nearly an order of magnitude quicker compared to the \TLBT and \IRKA approaches, as it eliminates the need for costly time-stepping integrators. We emphasize that the primary factor contributing to the computational time involved in constructing $\Gamma$ is the evaluation of some of the eigenvalues of the matrix pair $(\fA,\fE)$ as well as the expression \eqref{eqn:new:pse:set} for specific $\lapVar\in\C$. Nonetheless, there are instances where the spectral characteristics or eigenvalue bounds of $(\fA,\fE)$ are known in advance, which can be utilized to significantly accelerate the construction process of $\Gamma$. For example, in the case of the thermal block, it could be used the fact that all the eigenvalues are known to be negative real numbers, being the matrix a suitable discretization of a self-adjoint operator. Consequently, we could bypass these calculations, and according to our experiments, this would considerably reduce the offline stage from $10$ to $3.5$ seconds.

Despite the overwhelming numerical evidence, neither \TLBT nor \IRKA theoretically guarantee that a \ROM originating from an asymptotically stable \FOM will maintain its asymptotic stability. For further discussion and computational solutions, see \cite[Sec. 4.3]{Kur18} and \cite[Chap. 5]{gugercin_2008}. The \CIM-\LCS approach, on the other hand, does not encounter this problem since it does not utilize projection or time-stepping strategies. 

\begin{table}[t]
\centering
\renewcommand{\arraystretch}{1.5}
\begin{tabular}{lccccccc}
\toprule
&\multicolumn{3}{c}{\CIM-\LCS} & \multicolumn{2}{c}{\TLBT} & \multicolumn{2}{c}{\IRKA} \\
\cmidrule(lr){2-4} \cmidrule(lr){5-6} \cmidrule(lr){7-8}
&\multicolumn{2}{c}{Offline} & {Online} & {Offline} & {Online} & {Offline} & {Online} \\
\cmidrule(lr){2-3} 
&Const. $\Gamma$ & Assemble &  & &  & & \\
\midrule
Int. Sp. Sta. &$4.6\cdot10^{-1}$ & $1.3\cdot10^{-1}$ &  $5.4\cdot10^{-3}$ & $9.1\cdot10^{-1}$ & $6.5\cdot10^{-2}$ & $2.7\cdot10^0$ & $5.1\cdot10^{-2}$ \\
Therm. B. &$1.0\cdot10^{1}$ & $2.4\cdot10^{-1}$ & $1.7\cdot10^{-3}$ & $1.2\cdot10^0$ & $2.4\cdot10^{-2}$ & $5.3\cdot10^{0}$ & $9.0\cdot10^{-3}$ \\
\bottomrule
\end{tabular}
\caption{Computational times compared. The final time of interest is $T=100$ and the error at the final time is smaller than $\tol=10^{-5}$ for all the three compared methdology. Size of the \ROM for $\TLBT$ and $\IRKA$ is $\stateDimRed=27$ for the international space station and $\stateDimRed=20$ for the thermal block.}\label{tab:T1}
\end{table}

\subsubsection{Speeding up the contour construction exploiting \ROMs}\label{subsec:4.1 CIM-MOR}

As indicated in \Cref{tab:T1}, the primary computational challenge of \CIM-\LCS lies in forming the integration profile $\Gamma$. To expedite this process, one can initially construct a \ROM using conventional \LCS techniques like \BT or \IRKA, and subsequently apply the \CIM to the \ROM. This step significantly reduces computational expenses, as demonstrated in \Cref{tab:T2}, where the time required to develop $\Gamma$ becomes minimal in comparison to the \ROM construction. We refer to this approach as the \ROM-\CIM-\LCS to remark that before applying the \CIM-\LCS we construct a \ROM employing one of the traditional methods.

Nevertheless, this approach presents two significant drawbacks. Firstly, the Petrov-Galerkin projection, employed to obtain the \ROM, results in a spectrum and resolvent norm magnitude that differ from the original problem, and these can potentially be less favourable. For example, with a symmetric positive definite matrix, the spectrum is real, and the resolvent norm magnitude is inversely related to the spectrum distance. However, when a Petrov-Galerkin projection is applied to such a matrix, eigenvalues may gain imaginary components, and the reduced operator could become non-normal, leading to a large resolvent norm even when distanced from the spectrum. The second limitation is that this method does not extend seamlessly to parametric scenarios, unlike the inherently adaptable \CIM-\LCS.

\begin{table}[t]
\centering
\renewcommand{\arraystretch}{1.5}
\begin{tabular}{lcccc}
\toprule
&\multicolumn{4}{c}{\ROM-\CIM-\LCS}  \\
\cmidrule(lr){2-5} 
&\multicolumn{3}{c}{Offline} & {Online}  \\
\cmidrule(lr){2-4} 
&Construction \ROM &Construction $\Gamma$ & Assemble &  \\
\midrule
Int. Sp. Sta. & $1.2\cdot 10^{-1}$ & $8.5\cdot 10^{-2}$ &  $2.6\cdot 10^{-3}$ & $5.6\cdot 10^{-4}$ \\
Therm. B. &  $6.7\cdot 10^{-1}$ &  $5.3\cdot 10^{-3}$ & $6.0\cdot 10^{-2}$ & $8.5\cdot 10^{-4}$  \\
\bottomrule
\end{tabular}
\caption{The \ROM-\CIM-\LCS method with initial \ROM constructed via \BT. Time of interst $T=10$.}\label{tab:T2}
\end{table}

\subsection{Parametric linear control systems} \label{subsection:parametric}
\Cref{tab:T0bis} lists the benchmark problems considered to test the \CIM-\LCS.

\begin{table}[h!]
\centering
\renewcommand{\arraystretch}{1.5}
\begin{tabular}{lrccc}
\toprule
Example&$\stateDim$ & $\inpDim$ & $\outDim$& $\prmtrDim$ \\
\midrule
Parametric thermal block & $7488$& 1& 4 &1\\
Silicon nitride membrane & 60020 &1&2&4\\
\bottomrule
\end{tabular}
\caption{A list of examples with their dimensions ($\stateDim$), the number of inputs
($\inpDim$), outputs ($\outDim$), and parameters ($\prmtrDim$). These examples are taken from \url{  https://modelreduction.org/morwiki/Main_Page}.}\label{tab:T0bis}
\end{table}

\subsubsection{The one-parameter thermal block}
This is the one parameter variant of the thermal block problem presented in \Cref{sec:TB:ex}, see \cite{morwiki_thermalblock}. The parameter dependnce is in $\fA$ and reads as
\begin{equation*}
    \fA(\mu)\; \vcentcolon=\; \fA_1+\mu \fA_2, 	\quad\quad \mu  \in \prmtrSet\;\vcentcolon= \;[10^{-6}, 10^2],
\end{equation*}
where $\fA_1, \fA_2 \in \R^{\stateDim\times \stateDim}$ with $\stateDim = 7488$, are sparse, non-Hermitian matrices. The remaining matrices are the same as those referred to in \Cref{sec:TB:ex}. We define an accuracy threshold of $\tol=10^{-7}$ and design the integration contour $\Gamma$ with respect to the parameter values $\mu=100$, $T=25$, and $\Lambda=2$, which consequently yields $N=57$ quadrature points. With this configuration of $\Gamma$ for $\mu=100$ and $t\in[T,\Lambda T]$, the initial hypothesis outlined at the start of \Cref{sec3} holds true, allowing for the application of the $\CIM-\MOR$. This is achieved by executing \Cref{alg:1} across a discretized parameter domain $\Xi\subseteq\prmtrSet$, which consists of $1000$ points distributed logarithmically, ensuring an even spread over the parameter domain $\prmtrSet$ that spans multiple orders of magnitude. We consider the initial condition to belong to a subspace $\calF$ of dimension $\stateDimRed_{\calF}=3$ where the basis for $\calF$ is generated by orthonormalizing $3$ vectors, where each entry is generated by a random normal distribution. Finally, we set $\varepsilon=10^{-8}$, so that the reduction error is negligible with respect to the error generated by the \CIM time integration.
\begin{figure}[t]
	\centering{
	\begin{subfigure}[t]{0.48\textwidth}
	    \input{img/ParTB/Example_ParTB_1}
        \subcaption{Reduction error $\calE_{\rr}(\prmtr,\inp,\inSolSub)$ (see \eqref{eqn:err:red:par:CIM}) over $\Xi$ for the initial data in \eqref{eqn:ParTB:ID}.} \label{figParTBa}
	\end{subfigure}
    \hfill
    \begin{subfigure}[t]{0.48\textwidth}
	    \input{img/ParTB/Example_ParTB_2}
        \subcaption{Decay of the local error estimate $\tilde \Delta_j(\prmtr)$ with respect to the greedy algorithm iterations $J$ and target exit tollerance $\tilde \varepsilon_j$; see \eqref{eqn:local:err:est:tol}.} \label{figParTBb}
	\end{subfigure}
 }
	\caption{The one parameter thermal block benchmark for $T=25$, $\Lambda=2$, and $\mu\in[10^{-6},10^2]$.}
	\label{fig1:ParTB}%
\end{figure}
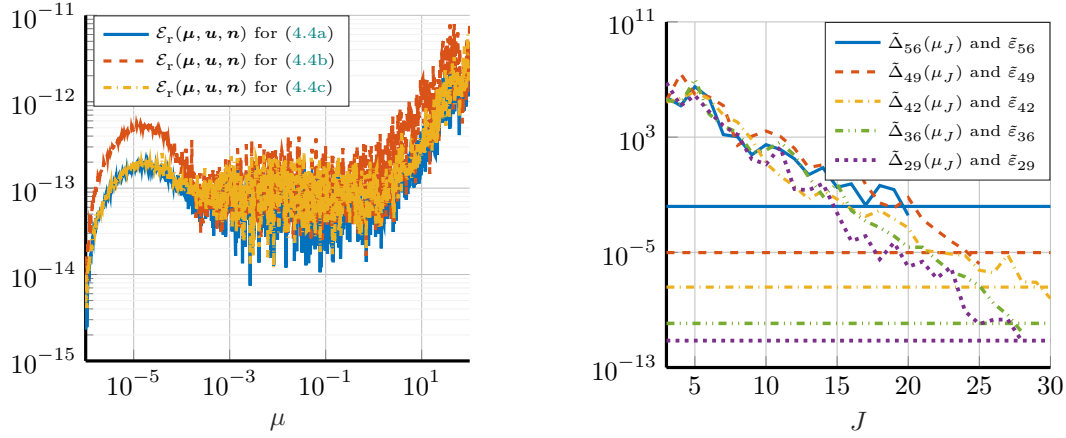%

After termination of \Cref{alg:1}, 
we test the \ROM given by the non-linear \CIM-\MOR for the 
whole parameter domain, the following inputs $\inp(t)$ and initial conditions $\stx_0 = \inMat \inSolSub$: 
\begin{subequations}\label{eqn:ParTB:ID}
    \begin{align}
\inp(t) \;=&\; \exp(-t),   & \inSolSub \;=&\; \begin{bmatrix}
    1&1&1
\end{bmatrix}^\T,  \label{eqn:ParTB:inp1}\\
\inp(t) \;=&\; t,         & \inSolSub \;=&\; \rand(\stateDimRed_{\calF},1), \label{eqn:ParTB:inp2}\\
\inp(t) \;=&\; \cos{\left(\frac{t}{5}\right)},         & \inSolSub \;=&\; \randn(\stateDimRed_{\calF},1). \label{eqn:ParTB:inp3}
    \end{align}
\end{subequations}
\Cref{figParTBa} presents the reduction error, as defined by \eqref{eqn:err:red:par:CIM}, over the parameter set $\Xi$ for the initial conditions and inputs specified by \eqref{eqn:ParTB:ID}. As anticipated, the reduction error remains consistently below the desired precision threshold $\varepsilon=10^{-8}$ across all three scenarios. Additionally, despite the stringent accuracy requirements, the evaluation of the \ROM, specifically the calculation in \eqref{eqn:out:par:CIM:red}, was approximately $10$ times quicker than the assessment of the \FOM, represented by the evaluation of \eqref{eqn:out:par:CIM}. 

Let us define the local error estimate and target exit tolerance at the quadrature point $j$ as
\begin{equation}\label{eqn:local:err:est:tol}
    \tilde \Delta_j(\prmtr)\;\vcentcolon=\;\frac{\Delta(\prmtr)}{(N-1)w_j},\quad \quad\tilde\varepsilon_j\;\vcentcolon=\;\frac{\varepsilon}{(N-1)w_j};
\end{equation}
The behaviour of these terms for some values of $j$ with respect to the iteration number $J$ of \Cref{alg:1} are reported in \Cref{figParTBb}.
\begin{table}[t]
\centering
\renewcommand{\arraystretch}{1.3} 
\begin{tabular}{@{} l c c c @{}}
\toprule
     & \textbf{Max} & \textbf{Mean} & \textbf{Min} \\ 
\midrule
\textbf{$\stateDimRed_j$} & 120 & 100  & 76  \\ 
\textbf{$J_{j,\max}$} & 30 & 25 &  19 \\ 
\bottomrule
\end{tabular}
\caption{The one parameter thermal block benchmark for $T=25$, $\Lambda=2$, and $\mu\in[10^{-6},10^2]$. Maximum, mean, and minimum value for: the size of the dense linear system solved to evaluate the \ROM \eqref{eqn:rstate:par:CIM} and the number of parameters selected $J_{j,\max}$ at iteration $j$ of \Cref{alg:1}.}\label{Tab:5}
\end{table}
In \Cref{Tab:5}, we present a summary of the maximum, average, and minimum sizes of the \ROM, evaluated across the quadrature points $j=1,\ldots,N-1$. Additionally, we report the same values for the total number of parameters selected by the greedy algorithm until the error estimate falls below the specified exit tolerance.

\subsubsection{Silicon nitride membrane}

The last numerical example we present is the silicon nitride membrane benchmark of \cite{morwiki}. This is a parametric problem of dimension $\prmtrDim=4$ with state space dimension $\stateDim=60020$ and input and output dimensions of, respectively, $\inpDim=1$ and $\outDim=2$. All matrices are real and the parametric dependence is in the state matrices $\fE(\prmtr)$ and $\fA(\prmtr)$ given by
\begin{align*}
\fE(\prmtr)\;\vcentcolon=&\;\fE_1+\mu_3\mu_2\fE_2,\quad \fA(\prmtr)\;\vcentcolon=\;\fA_1+\mu_1\fA_2+\mu_4\fA_3,\quad\text{for}\\
\prmtr\vcentcolon=&\begin{bmatrix}
        \mu_1&\mu_2&\mu_3&\mu_4
    \end{bmatrix}^{\T}\in\prmtrSet\;\vcentcolon=\;[2,5]\times[400,750]\times[3000,3200]\times[10,12].
\end{align*}
In the original benchmark, the time window of interest is $[0,0.04]$ for the specific discontinuous input function $\inp(t)=H(0.02-t)$ where $H(t)$ is the Heaviside function. Here we consider the same input function but for the time window of interest $[T,\Lambda T]$ with $T=0.03$ and $\Lambda=2$. We define an accuracy threshold of $\tol=10^{-7}$ and design the integration contour $\Gamma$ with respect to the parameter values $\prmtr=[5\;\;750\;\;3200\;\;12]^{\T}$, which yields $N=59$ quadrature points. With this configuration of $\Gamma$, the initial hypothesis outlined at the start of \Cref{sec3} holds true, allowing for the application of the \CIM-\MOR. This is achieved by executing \Cref{alg:1} across a discretized domain $\Xi\subseteq\prmtrSet$, which consists of a grid with $7$ Chebyshev distributed points in each parameter direction for a total of $|\Xi|=7^4=2401$ points. We consider the initial data to belong to a subspace $\calF$ of dimension $\stateDimRed_{\calF}=1$ where the base of $\calF$ is generated by normalizing $1$ vector, where each entry is generated by a random normal distribution. Ultimately, we chose $\varepsilon=10^{-6}$ despite the fact that the \FOM precision was $\varepsilon=10^{-7}$. This decision was made because, for values of $\varepsilon$ lower than this, the resulting $\tilde \varepsilon_j$ for certain $j$ was only marginally greater than the machine precision threshold, leading to convergence issues.

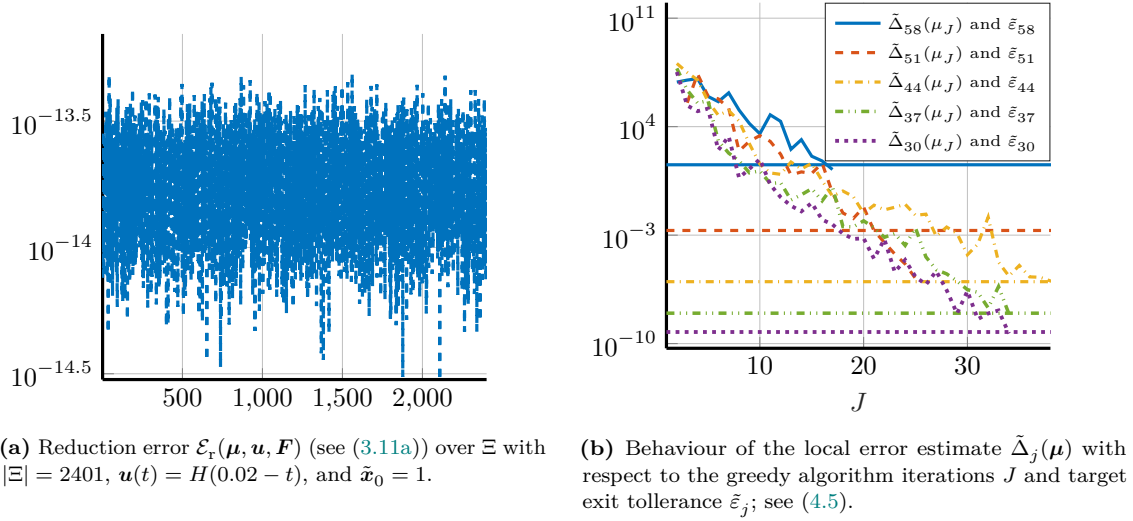
\begin{figure}[t]
	\centering{
	\begin{subfigure}[t]{0.48\textwidth}
	    \input{img/SNM/Example_SNM_1}
        \subcaption{Reduction error $\calE_{\rr}(\prmtr,\inp,\inMat)$ (see \eqref{eqn:err:red:par:CIM}) over $\Xi$ with $|\Xi|=2401$, $\inp(t)=H(0.02-t)$, and $\inSolSub=1$.} \label{figSNMa}
	\end{subfigure}
    \hfill
    \begin{subfigure}[t]{0.48\textwidth}
	    \input{img/SNM/Example_SNM_2}
        \subcaption{Behaviour of the local error estimate $\tilde \Delta_j(\prmtr)$ with respect to the greedy algorithm iterations $J$ and target exit tollerance $\tilde \varepsilon_j$; see \eqref{eqn:local:err:est:tol}.} \label{figSNMb}
	\end{subfigure}
 }
	\caption{Silicon nitride membrane benchmark for $T=0.03$, $\Lambda=2$, and $\prmtrDim=4$.}
	\label{fig1:SNM}%
\end{figure}%

Similarly to the parametric thermal problem, we report, in \Cref{figSNMa}, the reduction error \eqref{eqn:err:red:par:CIM} computed over the discrete parameter domain $\Xi$ for $\inp(t)=H(0.02-t)$ and $\inSolSub=1$. Also in \Cref{figSNMb} the decay of the local error estimate $\tilde \Delta_j(\prmtr)$ and the target exit tolerance $\tilde \varepsilon_j$ defined in \eqref{eqn:local:err:est:tol} are displayed for some $j$. We note that the reduction error is significantly smaller than the desired accuracy of $\varepsilon=10^{-6}$ across all parameters within the domain. This discrepancy may imply that the error estimate considerably overestimates the error's magnitude. Nonetheless, it is critical to emphasize that the error estimate encompasses all conceivable initial solutions within the subspace $\calF$ and all permissible input functions based on the construction of $\Gamma$, see discussion in \Cref{sec:CIM:ad:inp}. It is important to point out that \Cref{figSNMa} illustrates the error for a specific pairing of inputs and outputs. Hence, it may be possible that the ratio between the error estimate and the actual reduction error is consistently smaller for other combinations.
\begin{table}[t]
\centering
\renewcommand{\arraystretch}{1.3} 
\begin{tabular}{@{} l c c c @{}}
\toprule
     & \textbf{Max} & \textbf{Mean} & \textbf{Min} \\ 
\midrule
\textbf{$\stateDimRed_j$} & 78 & 64  & 32  \\ 
\textbf{$J_{j,\max}$} & 39 & 32 &  16 \\ 
\bottomrule
\end{tabular}
\caption{Silicon nitride membrane benchmark for $T=0.03$, $\Lambda=2$, and $\prmtrDim=4$. Maximum, mean, and minimum value for: the size of the dense linear system solved to evaluate the \ROM \eqref{eqn:rstate:par:CIM} and the number of parameters selected $J_{j,\max}$ at iteration $j$ of \Cref{alg:1}.}
\label{Tab:6}
\end{table}
We conclude with \Cref{Tab:6}, which displays the maximum, average, and minimum dimensions of the \ROM as evaluated across the quadrature points $j=1,\ldots,N-1$. Additionally, it shows the maximum, average, and minimum number of parameters selected by the greedy algorithm. It is noteworthy that the average computation time for evaluating the \FOM, i.e., \eqref{eqn:out:par:CIM}, across the parameter domain $\Xi$ was $21$ seconds. In contrast, the evaluation of \eqref{eqn:out:par:CIM:red} using the \CIM-\MOR approach for the \ROM required only $8\cdot10^{-3}$ seconds on average, achieving a speed-up factor of approximately $2689$, while maintaining an accuracy level of at least $10^{-6}$ concerning the error measure \eqref{eqn:err:red:par:CIM}. 

\section{Conclusions}\label{sec5}
We presented a novel Petrov-Galerkin projection-based non-linear \MOR strategy for parametric linear control systems based on the use of contour integral methods. This methodology is shown to be able to approximate accurately and efficiently the output function for a wide range of input functions, initial conditions, and uniformly over a prescribed time-window for the parameter domain of interest. For each quadrature point used for the approximation of the contour integral, we construct the projection spaces via a weak-greedy strategy, i.e., an algorithm driven by a rigorous error estimator. We demonstrate that, using the proposed projection spaces, the Hermite interpolation conditions corresponding to the parameters chosen by the greedy algorithm are satisfied by the Laplace transform of the reduced output function at the quadrature points. Numerical experiments on a variety of benchmark problems demonstrate the effectiveness of our methodology. Comparison with state-of-the-art methodology for the non-parametric case is extensively discussed as well as conditions under which the proposed methodology may struggle. 
In summary, the proposed methodology offers a robust and scalable instrument for the fast evaluation of parametric input–output maps in parametric linear control systems with varying initial conditions. 

\subsection*{Acknowledgments}
Funded by Deutsche Forschungsgemeinschaft (DFG, German Research Foundation) under Germany's Excellence Strategy - EXC 2075 – 390740016. We acknowledge the support by the Stuttgart Center for Simulation Science (SimTech). This work is part of the Society for Industrial and Applied Mathematics (SIAM) Postdoctoral Support Program, which is funded by contributions to the SIAM Postdoctoral Support Fund, established by a gift from Drs. Martin Golubitsky and Barbara Keyfitz. MM also acknowledges funding from the BMBF (grant no.~05M22VSA) and is grateful to TT.-Prof. Dr. Benjamin Unger for his support in the SIAM Postdoctoral Support Program application. The work of SG is supported in part by the US National Science Foundation under grant DMS-2411141.

\bibliographystyle{plain-doi}
\bibliography{journalabbr,General_Literature}    

\end{document}

%% file: def.tex
\newcommand{\R}{\ensuremath\mathbb{R}}
\newcommand{\C}{\ensuremath\mathbb{C}}

\newcommand{\smoothFunctions}[3][]{\ifthenelse{\equal{#1}{}}{\mathcal{C}^{#2}}{\mathcal{C}_{#1}^{#2}}(#3)}

\newcommand{\dist}[1][]{\ifthenelse{\equal{#1}{}}{\mathbb{D}}{#1_{\mathbb{D}}}}

\newcommand{\prmtr}{\boldsymbol{\mu}}
\newcommand{\prmtrSet}{\calP}
\newcommand{\prmtrDim}{d}

\newcommand{\T}{\top} 

\newcommand{\integrate}[1]{\mathrm{d}#1}

\newcommand{\dz}{\integrate{z}}

\DeclareMathOperator{\ee}{e}

\DeclareMathOperator{\range}{img}

\DeclareMathOperator{\real}{Re}
\DeclareMathOperator{\imag}{Im}

\newcommand{\calE}{\mathcal{E}}
\newcommand{\calF}{\mathcal{F}}

\newcommand{\calH}{\mathcal{H}}

\newcommand{\calJ}{\mathcal{J}}

\newcommand{\calO}{\mathcal{O}}
\newcommand{\calP}{\mathcal{P}}

\newcommand{\calR}{\mathcal{R}}

\newcommand{\calU}{\mathcal{U}}

\newcommand{\stx}{\ensuremath{\bm{x}}}
\newcommand{\redstx}{\ensuremath{\bm{x}}_{\mathrm{r}}}

\newcommand{\stateDim}{n}

\newcommand{\stateDimRed}{r}
\newcommand{\inp}{\ensuremath{\bm{u}}}
\newcommand{\inpDim}{m}
\newcommand{\out}{\ensuremath{\bm{y}}}
\newcommand{\redout}{\ensuremath{\bm{y}}_{\mathrm{r}}}
\newcommand{\outDim}{p}

\newcommand{\ff}{\ensuremath{\bm{f}}}

\newcommand{\fn}{\ensuremath{\bm{n}}}

\newcommand{\fp}{\ensuremath{\bm{p}}}
\newcommand{\fq}{\ensuremath{\bm{q}}}
\newcommand{\fr}{\ensuremath{\bm{r}}}

\newcommand{\fv}{\ensuremath{\bm{v}}}
\newcommand{\fw}{\ensuremath{\bm{w}}}

\newcommand{\fA}{\ensuremath{\bm{A}}}
\newcommand{\fB}{\ensuremath{\bm{B}}}
\newcommand{\fC}{\ensuremath{\bm{C}}}
\newcommand{\fD}{\ensuremath{\bm{D}}}
\newcommand{\fE}{\ensuremath{\bm{E}}}
\newcommand{\fF}{\ensuremath{\bm{F}}}
\newcommand{\fG}{\ensuremath{\bm{G}}}

\newcommand{\fI}{\ensuremath{\bm{I}}}

\newcommand{\fK}{\ensuremath{\bm{K}}}

\newcommand{\fM}{\ensuremath{\bm{M}}}

\newcommand{\fS}{\ensuremath{\bm{S}}}

\newcommand{\fV}{\ensuremath{\bm{V}}}
\newcommand{\fW}{\ensuremath{\bm{W}}}
\newcommand{\fX}{\ensuremath{\bm{X}}}

\newcommand{\zeroVec}{\mathbf{0}}

\newcommand{\abbr}[1]{\textsf{#1}\xspace}
\newcommand{\FOM}{\abbr{FOM}}
\newcommand{\FOMs}{\abbr{FOMs}}
\newcommand{\ROM}{\abbr{ROM}}
\newcommand{\ROMs}{\abbr{ROMs}}
\newcommand{\MOR}{\abbr{MOR}}

\newcommand{\SCM}{\abbr{SCM}}

\newcommand{\PDEs}{\abbr{PDEs}}

\newcommand{\RBM}{\abbr{RBM}}

\newcommand{\CIM}{\abbr{CIM}}
\newcommand{\CIMs}{\abbr{CIMs}}

\newcommand{\LTI}{\abbr{LTI}}

\newcommand{\RB}{\abbr{RB}}
\newcommand{\POD}{\abbr{POD}}
\newcommand{\CS}{\abbr{CS}}
\newcommand{\IRKA}{\abbr{IRKA}}
\newcommand{\BT}{\abbr{BT}}
\newcommand{\TLBT}{\abbr{TLBT}}
\newcommand{\LCS}{\abbr{LCS}}

\newcommand{\expmv}{\abbr{expmv}}

\newcommand{\lapOp}{\gleMat{L}}
\newcommand{\lapVar}{z}
\newcommand{\gleMat}[1]{\mathbf{\mathscr{#1}}}
\newcommand{\imagunit}{{\bf i}}
\newcommand{\domega}{\integrate{\omega}}

\newcommand{\intvar}{s}
\newcommand{\dintvar}{\integrate{\intvar}}
\newcommand{\tol}{\abbr{tol}}
\newcommand{\source}{\ff}
\newcommand{\intVarTwo}{g}
\newcommand{\eps}{\rm eps}

\usepackage{mathtools}
\usepackage{hyperref}
\usepackage{cleveref}

\newcommand{\rand}{\texttt{rand}}
\newcommand{\randn}{\texttt{randn}}
\newcommand{\orth}{\texttt{orth}}
\newcommand{\matlab}{{\sc Matlab~}}

\newcommand{\res}{{\fr}}
\newcommand{\rr}{\mathrm{r}}
\newcommand{\inMat}{\fF}
\newcommand{\inSolSub}{\tilde \stx_0}

\newcommand{\redfA}{\ensuremath{\bm{A}}_{\mathrm{r}}}
\newcommand{\redfB}{\ensuremath{\bm{B}}_{\mathrm{r}}}
\newcommand{\redfC}{\ensuremath{\bm{C}}_{\mathrm{r}}}

\newcommand{\redfE}{\ensuremath{\bm{E}}_{\mathrm{r}}}

%% file: img/Fig1.tex
\begin{tikzpicture}
\begin{axis}[
    width=\imageWidth,
	height=\imageHeight,
	scale only axis,
    axis equal,
    axis lines=middle,
    xmin=-0.9, xmax=0.2,
    ymin=-1.5, ymax=1.5,
    xlabel={$\real(\lapVar)$}, ylabel={$\imagunit\imag(\lapVar)$},
    ticks=none
]

\def\a{1}    
\def\b{0.5}  

\addplot [color=mycolor1, line width=\lineWidth,
    domain=-90:90,
    samples=300,
] ({\a*cos(x)-0.7}, {\b*sin(x)});

\def\xend{-0.7}
\def\yend{-0.5}


\def\L{1.5}  
\def\angle{-120}

\addplot[-, color=mycolor1, line width=\lineWidth] coordinates {
    (\xend,\yend)
    ({\xend + \L*sin(\angle)}, {\yend + \L*cos(\angle)})
};

\def\xend{-0.7}
\def\yend{0.5}

\def\L{1.5}  
\def\angle{-60}

\addplot[-, color=mycolor1, line width=\lineWidth] coordinates {
    (\xend,\yend)
    ({\xend + \L*sin(\angle)}, {\yend + \L*cos(\angle)})
};

\end{axis}
\end{tikzpicture}

%% file: img/Fig2.tex
%
\begin{tikzpicture}

\begin{axis}[%
width=\imageWidth,
	height=\imageHeight,
	scale only axis,
	scaled ticks=false,
	grid=both,
	grid style={line width=.1pt, draw=gray!10},
	major grid style={line width=.2pt,draw=gray!50},
	axis lines*=left,
	axis line style={line width=\lineWidth},
xmin=-4,
xmax=3,
xlabel style={font=\color{white!15!black}},
xlabel={$\real(\lapVar)$},
ymin=-20,
ymax=20,
ylabel style={font=\color{white!15!black}},
ylabel={$ \imagunit \imag(\lapVar)$},
	axis background/.style={fill=white},
	legend style={%
		legend cell align=left, 
		align=left, 
		font=\tiny,
		draw=white!15!black,
		at={(1.00,1.00)},
		anchor=north east,},
]
\addplot [color=mycolor1,line width=\lineWidth]
  table [x index=0, y index=1, col sep=comma]{img/DataCSV/BG/Map.csv};

\addlegendentry{$\Gamma$ }

\addplot [color=mycolor3,dash dot,line width=1.2*\lineWidth]
  table [x index=2, y index=3, col sep=comma]{img/DataCSV/BG/Map.csv};

\addlegendentry{$\Gamma_{right}$ }

\addplot [color=mycolor2,dashed,line width=1.2*\lineWidth]
  table [x index=4, y index=5, col sep=comma]{img/DataCSV/BG/Map.csv};

\addlegendentry{$\Gamma_{left}$ }

\end{axis}
\end{tikzpicture}%

%% file: img/ISS/Example_ISS_1.tex
%
\begin{tikzpicture}

\begin{axis}[%
	width=0.64*\imageWidth,
	height=\imageHeight,
	scale only axis,
	grid=both,
	grid style={line width=.1pt, draw=gray!10},
	major grid style={line width=.2pt,draw=gray!50},
	axis lines*=left,
	axis line style={line width=\lineWidth},
    scaled x ticks = true,        
xmin=1,
xmax=1.5e3,
xlabel style={font=\color{white!15!black}},
xlabel={N},
 xticklabel style={
    /pgf/number format/sci,
    /pgf/number format/precision=2
  },
xtick={5e2,1e3},  
ymode=log,
ymin=5e-9,
ymax=1e5,
yminorticks=true,
ylabel style={font=\color{white!15!black}},
	axis background/.style={fill=white},
	legend style={%
		legend cell align=left, 
		align=left, 
		font=\tiny,
		draw=white!15!black,
		at={(1.0,1.0)},
		anchor=north east,},
]

\addplot [color=mycolor1, line width=\lineWidth, forget plot]
 table [x index=0, y index=1, col sep=comma]{img/DataCSV/ISS/Input1/1.000000e-04.csv};

\addplot [color=mycolor1, dotted, line width=\lineWidth]
  table[row sep=crcr]{%
1	 1e-4\\
3.8e3 1e-4\\
};
\addlegendentry{$\tol=10^{-4}$}


\addplot [color=mycolor3, line width=\lineWidth, forget plot]
 table [x index=0, y index=1, col sep=comma]{img/DataCSV/ISS/Input1/1.000000e-03.csv};
 
\addplot [color=mycolor3, densely dotted, line width=\lineWidth]
  table[row sep=crcr]{%
1	1e-03\\
3.8e3	1e-03\\
};
\addlegendentry{$\tol=10^{-3}$}


\addplot [color=mycolor4, line width=\lineWidth, forget plot]
 table [x index=0, y index=1, col sep=comma]{img/DataCSV/ISS/Input1/1.000000e-02.csv};

\addplot [color=mycolor4, loosely dashdotted, line width=\lineWidth]
  table[row sep=crcr]{%
1	1e-02\\
3.8e3	1e-02\\
};
\addlegendentry{$\tol=10^{-2}$}


\addplot [color=mycolor2, line width=\lineWidth, forget plot]
 table [x index=0, y index=1, col sep=comma]{img/DataCSV/ISS/Input1/1.000000e-01.csv};
 
\addplot [color=mycolor2, densely dashdotted, line width=\lineWidth]
  table[row sep=crcr]{%
1	1e-1\\
3.8e3	1e-1\\
};
\addlegendentry{$\tol=10^{-1}$}

\addplot [color=mycolor5, line width=\lineWidth, only marks, mark=o, mark options={solid, mycolor5}]
table [x index=0, y index=1, col sep=comma]{img/DataCSV/ISS/Input1/OptN.csv};
\addlegendentry{$\tilde N$ (see \eqref{eqn:opt:num:QP})}
\end{axis}
\end{tikzpicture}%

%% file: img/ISS/Example_ISS_2.tex
%
\begin{tikzpicture}

\begin{axis}[%
	width=0.64*\imageWidth,
	height=\imageHeight,
	scale only axis,
	grid=both,
	grid style={line width=.1pt, draw=gray!10},
	major grid style={line width=.2pt,draw=gray!50},
	axis lines*=left,
	axis line style={line width=\lineWidth},
    scaled x ticks = true,        
xmin=1,
xmax=1.5e3,
xlabel style={font=\color{white!15!black}},
xlabel={N},
 xticklabel style={
    /pgf/number format/sci,
    /pgf/number format/precision=2
  },
xtick={5e2,1e3},  
ymode=log,
ymin=5e-9,
ymax=1e5,
yminorticks=true,
ylabel style={font=\color{white!15!black}},
	axis background/.style={fill=white},
	legend style={%
		legend cell align=left, 
		align=left, 
		font=\tiny,
		draw=white!15!black,
		at={(1.0,1.0)},
		anchor=north east,},
]

\addplot [color=mycolor1, line width=\lineWidth, forget plot]
 table [x index=0, y index=1, col sep=comma]{img/DataCSV/ISS/Input2/1.000000e-04.csv};

\addplot [color=mycolor1, dotted, line width=\lineWidth]
  table[row sep=crcr]{%
1	 1e-4\\
3.8e3 1e-4\\
};


\addplot [color=mycolor3, line width=\lineWidth, forget plot]
 table [x index=0, y index=1, col sep=comma]{img/DataCSV/ISS/Input2/1.000000e-03.csv};
 
\addplot [color=mycolor3, densely dotted, line width=\lineWidth]
  table[row sep=crcr]{%
1	1e-03\\
3.8e3	1e-03\\
};


\addplot [color=mycolor4, line width=\lineWidth, forget plot]
 table [x index=0, y index=1, col sep=comma]{img/DataCSV/ISS/Input2/1.000000e-02.csv};

\addplot [color=mycolor4, loosely dashdotted, line width=\lineWidth]
  table[row sep=crcr]{%
1	1e-2\\
3.8e3	1e-2\\
};


\addplot [color=mycolor2, line width=\lineWidth, forget plot]
 table [x index=0, y index=1, col sep=comma]{img/DataCSV/ISS/Input2/1.000000e-01.csv};
 
\addplot [color=mycolor2, densely dashdotted, line width=\lineWidth]
  table[row sep=crcr]{%
1	1e-1\\
3.8e3	1e-1\\
};

\addplot [color=mycolor5, line width=\lineWidth, only marks, mark=o, mark options={solid, mycolor5}]
table [x index=0, y index=1, col sep=comma]{img/DataCSV/ISS/Input2/OptN.csv};
\end{axis}
\end{tikzpicture}%

%% file: img/ISS/Example_ISS_3.tex
%
\begin{tikzpicture}

\begin{axis}[%
width=0.64*\imageWidth,
	height=\imageHeight,
	scale only axis,
	scaled ticks=false,
	grid=both,
	grid style={line width=.1pt, draw=gray!10},
	major grid style={line width=.2pt,draw=gray!50},
	axis lines*=left,
	axis line style={line width=\lineWidth},
xmin=-1,
xmax=0.35,
xlabel style={font=\color{white!15!black}},
xlabel={$\real(\lapVar)$},
ymin=-60,
ymax=60,
ylabel style={font=\color{white!15!black}},
ylabel={$ \imagunit \imag(\lapVar)$},
	axis background/.style={fill=white},
	legend style={%
		legend cell align=left, 
		align=left, 
		font=\tiny,
		draw=white!15!black,
		at={(0.60,0.65)},
		anchor=north east,},
]

\addplot  [color=mycolor5,line width=\lineWidth, only marks, mark=x, mark options={solid, mycolor5}]
 table [x index=0, y index=1, col sep=comma]{img/DataCSV/ISS/SpectrumAE.csv};
\addlegendentry{$\Sigma(\fA)$}

\addplot [color=mycolor1,line width=\lineWidth]
  table [x index=0, y index=1, col sep=comma]{img/DataCSV/ISS/Gamma.csv};

\addlegendentry{$\Gamma$}

\addplot [color=mycolor2,dash dot,line width=1.2*\lineWidth]
  table [x index=2, y index=3, col sep=comma]{img/DataCSV/ISS/Gamma.csv};

\addlegendentry{$\Gamma_{left}$ }

\addplot [color=mycolor3,dashed,line width=1.2*\lineWidth]
  table [x index=4, y index=5, col sep=comma]{img/DataCSV/ISS/Gamma.csv};

\addlegendentry{$\Gamma_{right}$ }

\end{axis}
\end{tikzpicture}%

%% file: img/TB/Example_TB_1.tex
%
\begin{tikzpicture}

\begin{axis}[%
	width=0.64*\imageWidth,
	height=\imageHeight,
	scale only axis,
	grid=both,
	grid style={line width=.1pt, draw=gray!10},
	major grid style={line width=.2pt,draw=gray!50},
	axis lines*=left,
	axis line style={line width=\lineWidth},
    scaled x ticks = true,        
xmin=1,
xmax=80,
xlabel style={font=\color{white!15!black}},
xlabel={N},  
ymode=log,
ymin=1e-10,
ymax=1e6,
yminorticks=true,
ylabel style={font=\color{white!15!black}},
	axis background/.style={fill=white},
	legend style={%
		legend cell align=left, 
		align=left, 
		font=\tiny,
		draw=white!15!black,
		at={(1.0,1.0)},
		anchor=north east,},
]

\addplot [color=mycolor1, line width=\lineWidth, forget plot]
 table [x index=0, y index=1, col sep=comma]{img/DataCSV/TB/Input1/1.000000e-08.csv};
 
\addplot [color=mycolor1, dotted, line width=\lineWidth]
  table[row sep=crcr]{%
1	1e-08\\
2200	1e-08\\
};
\addlegendentry{$\tol=10^{-8}$}


\addplot [color=mycolor3, line width=\lineWidth, forget plot]
 table [x index=0, y index=1, col sep=comma]{img/DataCSV/TB/Input1/1.000000e-06.csv};
 
\addplot [color=mycolor3, densely dotted, line width=\lineWidth]
  table[row sep=crcr]{%
1	1e-06\\
2200	1e-06\\
};
\addlegendentry{$\tol=10^{-6}$}


\addplot [color=mycolor4, line width=\lineWidth, forget plot]
 table [x index=0, y index=1, col sep=comma]{img/DataCSV/TB/Input1/1.000000e-04.csv};

\addplot [color=mycolor4, loosely dashdotted, line width=\lineWidth]
  table[row sep=crcr]{%
1	0.0001\\
2200	0.0001\\
};
\addlegendentry{$\tol=10^{-4}$}


\addplot [color=mycolor2, line width=\lineWidth, forget plot]
 table [x index=0, y index=1, col sep=comma]{img/DataCSV/TB/Input1/1.000000e-02.csv};
 
\addplot [color=mycolor2, densely dashdotted, line width=\lineWidth]
  table[row sep=crcr]{%
1	0.01\\
2200	0.01\\
};
\addlegendentry{$\tol=10^{-2}$}

\addplot [color=mycolor5, line width=\lineWidth, only marks, mark=o, mark options={solid, mycolor5}]
table [x index=0, y index=1, col sep=comma]{img/DataCSV/TB/Input1/OptN.csv};
\addlegendentry{$\tilde N$ (see \eqref{eqn:opt:num:QP})}

\end{axis}
\end{tikzpicture}%

%% file: img/TB/Example_TB_2.tex
%
\begin{tikzpicture}

\begin{axis}[%
	width=0.64*\imageWidth,
	height=\imageHeight,
	scale only axis,
	grid=both,
	grid style={line width=.1pt, draw=gray!10},
	major grid style={line width=.2pt,draw=gray!50},
	axis lines*=left,
	axis line style={line width=\lineWidth},
    scaled x ticks = true,        
xmin=1,
xmax=80,
xlabel style={font=\color{white!15!black}},
xlabel={N},
ymode=log,
ymin=1e-14,
ymax=1e6,
yminorticks=true,
ylabel style={font=\color{white!15!black}},
	axis background/.style={fill=white},
	legend style={%
		legend cell align=left, 
		align=left, 
		font=\tiny,
		draw=white!15!black,
		at={(1.0,1.0)},
		anchor=north east,},
]

\addplot [color=mycolor1, line width=\lineWidth, forget plot]
 table [x index=0, y index=1, col sep=comma]{img/DataCSV/TB/Input2/1.000000e-08.csv};

\addplot [color=mycolor1, dotted, line width=\lineWidth]
  table[row sep=crcr]{%
1	1e-8\\
2200	1e-8\\
};


\addplot [color=mycolor3, line width=\lineWidth, forget plot]
 table [x index=0, y index=1, col sep=comma]{img/DataCSV/TB/Input2/1.000000e-06.csv};
 
\addplot [color=mycolor3, densely dotted, line width=\lineWidth]
  table[row sep=crcr]{%
1	1e-06\\
2200	1e-06\\
};


\addplot [color=mycolor4, line width=\lineWidth, forget plot]
 table [x index=0, y index=1, col sep=comma]{img/DataCSV/TB/Input2/1.000000e-04.csv};

\addplot [color=mycolor4, loosely dashdotted, line width=\lineWidth]
  table[row sep=crcr]{%
1	0.0001\\
2200	0.0001\\
};


\addplot [color=mycolor2, line width=\lineWidth, forget plot]
 table [x index=0, y index=1, col sep=comma]{img/DataCSV/TB/Input2/1.000000e-02.csv};
 
\addplot [color=mycolor2, densely dashdotted, line width=\lineWidth]
  table[row sep=crcr]{%
1	0.01\\
2200	0.01\\
};

\addplot [color=mycolor5, line width=\lineWidth, only marks, mark=o, mark options={solid, mycolor5}]
table [x index=0, y index=1, col sep=comma]{img/DataCSV/TB/Input2/OptN.csv};
\end{axis}
\end{tikzpicture}%

%% file: img/TB/Example_TB_3.tex
%
\begin{tikzpicture}

\begin{axis}[%
width=0.64*\imageWidth,
	height=\imageHeight,
	scale only axis,
	scaled ticks=false,
	grid=both,
	grid style={line width=.1pt, draw=gray!10},
	major grid style={line width=.2pt,draw=gray!50},
	axis lines*=left,
	axis line style={line width=\lineWidth},
xmin=-1.5,
xmax=0.25,
xlabel style={font=\color{white!15!black}},
xlabel={$\real(\lapVar)$},
ymin=-1.2,
ymax=1.2,
ylabel style={font=\color{white!15!black}},
ylabel={$ \imagunit \imag(\lapVar)$},
	axis background/.style={fill=white},
	legend style={%
		legend cell align=left, 
		align=left, 
		font=\tiny,
		draw=white!15!black,
		at={(0.62,0.90)},
		anchor=north east,},
]

\addplot  [color=mycolor5,line width=\lineWidth, only marks, mark=x, mark options={solid, mycolor5}]
 table [x index=0, y index=1, col sep=comma]{img/DataCSV/TB/SpectrumAE.csv};
\addlegendentry{$\Sigma(\fA,\fE)$}

\addplot [color=mycolor1,line width=\lineWidth]
  table [x index=0, y index=1, col sep=comma]{img/DataCSV/TB/Gamma.csv};

\addlegendentry{$\Gamma$}

\addplot [color=mycolor2,dash dot,line width=1.2*\lineWidth]
  table [x index=2, y index=3, col sep=comma]{img/DataCSV/TB/Gamma.csv};

\addlegendentry{$\Gamma_{left}$ }

\addplot [color=mycolor3,dashed,line width=1.2*\lineWidth]
  table [x index=4, y index=5, col sep=comma]{img/DataCSV/TB/Gamma.csv};

  \addlegendentry{$\Gamma_{right}$ }

\end{axis}
\end{tikzpicture}%

%% file: img/BG/Example_BG_1.tex
%
\begin{tikzpicture}

\begin{axis}[%
	width=0.64*\imageWidth,
	height=\imageHeight,
	scale only axis,
	grid=both,
	grid style={line width=.1pt, draw=gray!10},
	major grid style={line width=.2pt,draw=gray!50},
	axis lines*=left,
	axis line style={line width=\lineWidth},
    scaled x ticks = true,        
xmin=1,
xmax=60,
xlabel style={font=\color{white!15!black}},
xlabel={N}, 
ymode=log,
ymin=1e-10,
ymax=1e6,
yminorticks=true,
ylabel style={font=\color{white!15!black}},
	axis background/.style={fill=white},
	legend style={%
		legend cell align=left, 
		align=left, 
		font=\tiny,
		draw=white!15!black,
		at={(1.0,1.0)},
		anchor=north east,},
]

\addplot [color=mycolor1, line width=\lineWidth, forget plot]
 table [x index=0, y index=1, col sep=comma]{img/DataCSV/BG/Input1/1.000000e-08.csv};

\addplot [color=mycolor1, dotted, line width=\lineWidth]
  table[row sep=crcr]{%
1	 1e-8\\
2200 1e-8\\
};
\addlegendentry{$\tol=10^{-8}$}


\addplot [color=mycolor3, line width=\lineWidth, forget plot]
 table [x index=0, y index=1, col sep=comma]{img/DataCSV/BG/Input1/1.000000e-06.csv};
 
\addplot [color=mycolor3, densely dotted, line width=\lineWidth]
  table[row sep=crcr]{%
1	1e-06\\
2200	1e-06\\
};
\addlegendentry{$\tol=10^{-6}$}


\addplot [color=mycolor4, line width=\lineWidth, forget plot]
 table [x index=0, y index=1, col sep=comma]{img/DataCSV/BG/Input1/1.000000e-04.csv};

\addplot [color=mycolor4, loosely dashdotted, line width=\lineWidth]
  table[row sep=crcr]{%
1	0.0001\\
2200	0.0001\\
};
\addlegendentry{$\tol=10^{-4}$}


\addplot [color=mycolor2, line width=\lineWidth, forget plot]
 table [x index=0, y index=1, col sep=comma]{img/DataCSV/BG/Input1/1.000000e-02.csv};
 
\addplot [color=mycolor2, densely dashdotted, line width=\lineWidth]
  table[row sep=crcr]{%
1	0.01\\
2200	0.01\\
};
\addlegendentry{$\tol=10^{-2}$}

\addplot [color=mycolor5, line width=\lineWidth, only marks, mark=o, mark options={solid, mycolor5}]
table [x index=0, y index=1, col sep=comma]{img/DataCSV/BG/Input1/OptN.csv};
\addlegendentry{$\tilde N$ (see \eqref{eqn:opt:num:QP})}
\end{axis}
\end{tikzpicture}%

%% file: img/BG/Example_BG_2.tex
%
\begin{tikzpicture}

\begin{axis}[%
	width=0.64*\imageWidth,
	height=\imageHeight,
	scale only axis,
	grid=both,
	grid style={line width=.1pt, draw=gray!10},
	major grid style={line width=.2pt,draw=gray!50},
	axis lines*=left,
	axis line style={line width=\lineWidth},
    scaled x ticks = true,        
xmin=1,
xmax=60,
xlabel style={font=\color{white!15!black}},
xlabel={N},
ymode=log,
ymin=1e-10,
ymax=1e5,
yminorticks=true,
ylabel style={font=\color{white!15!black}},
	axis background/.style={fill=white},
	legend style={%
		legend cell align=left, 
		align=left, 
		font=\tiny,
		draw=white!15!black,
		at={(1.0,1.0)},
		anchor=north east,},
]

\addplot [color=mycolor1, line width=\lineWidth, forget plot]
 table [x index=0, y index=1, col sep=comma]{img/DataCSV/BG/Input2/1.000000e-08.csv};

\addplot [color=mycolor1, dotted, line width=\lineWidth]
  table[row sep=crcr]{%
1	 1e-8\\
2200 1e-8\\
};


\addplot [color=mycolor3, line width=\lineWidth, forget plot]
 table [x index=0, y index=1, col sep=comma]{img/DataCSV/BG/Input2/1.000000e-06.csv};
 
\addplot [color=mycolor3, densely dotted, line width=\lineWidth]
  table[row sep=crcr]{%
1	1e-06\\
2200	1e-06\\
};


\addplot [color=mycolor4, line width=\lineWidth, forget plot]
 table [x index=0, y index=1, col sep=comma]{img/DataCSV/BG/Input2/1.000000e-04.csv};

\addplot [color=mycolor4, loosely dashdotted, line width=\lineWidth]
  table[row sep=crcr]{%
1	0.0001\\
2200	0.0001\\
};


\addplot [color=mycolor2, line width=\lineWidth, forget plot]
 table [x index=0, y index=1, col sep=comma]{img/DataCSV/BG/Input2/1.000000e-02.csv};
 
\addplot [color=mycolor2, densely dashdotted, line width=\lineWidth]
  table[row sep=crcr]{%
1	0.01\\
2200	0.01\\
};

\addplot [color=mycolor5, line width=\lineWidth, only marks, mark=o, mark options={solid, mycolor5}]
table [x index=0, y index=1, col sep=comma]{img/DataCSV/BG/Input2/OptN.csv};
\end{axis}
\end{tikzpicture}%

%% file: img/BG/Example_BG_3.tex
%
\begin{tikzpicture}

\begin{axis}[%
width=0.64*\imageWidth,
	height=\imageHeight,
	scale only axis,
	scaled ticks=false,
	grid=both,
	grid style={line width=.1pt, draw=gray!10},
	major grid style={line width=.2pt,draw=gray!50},
	axis lines*=left,
	axis line style={line width=\lineWidth},
xmin=-3,
xmax=4,
xlabel style={font=\color{white!15!black}},
xlabel={$\real(\lapVar)$},
ymin=-10,
ymax=10,
ylabel style={font=\color{white!15!black}},
ylabel={$ \imagunit \imag(\lapVar)$},
	axis background/.style={fill=white},
	legend style={%
		legend cell align=left, 
		align=left, 
		font=\tiny,
		draw=white!15!black,
		at={(1.00,1.00)},
		anchor=north east,},
]

\addplot  [color=mycolor5,line width=\lineWidth, only marks, mark=x, mark options={solid, mycolor5}]
 table [x index=0, y index=1, col sep=comma]{img/DataCSV/BG/SpectrumAE.csv};
\addlegendentry{$\Sigma(\fA)$}

\addplot [color=mycolor1,line width=\lineWidth]
  table [x index=0, y index=1, col sep=comma]{img/DataCSV/BG/Gamma.csv};

 \addlegendentry{$\Gamma$ }

\addplot [color=mycolor2,dash dot,line width=1.2*\lineWidth]
  table [x index=2, y index=3, col sep=comma]{img/DataCSV/BG/Gamma.csv};

\addlegendentry{$\Gamma_{left}$ }

\addplot [color=mycolor3,dashed,line width=1.2*\lineWidth]
  table [x index=4, y index=5, col sep=comma]{img/DataCSV/BG/Gamma.csv};

  \addlegendentry{$\Gamma_{right}$ }

\end{axis}
\end{tikzpicture}%

%% file: img/ParTB/Example_ParTB_1.tex
%
\begin{tikzpicture}

\begin{axis}[%
width=\imageWidth,
	height=\imageHeight,
	scale only axis,
	grid=both,
	grid style={line width=.1pt, draw=gray!10},
	major grid style={line width=.2pt,draw=gray!50},
	axis lines*=left,
	axis line style={line width=\lineWidth},
    scaled x ticks = true,        
xmode=log,
xmin=1e-06,
xmax=100,
xlabel style={font=\color{white!15!black}},
xlabel={$\mu$}, 
ymode=log,
ymin=1e-15,
ymax=1e-11,
yminorticks=true,
ylabel style={font=\color{white!15!black}},
	axis background/.style={fill=white},
	legend style={%
		legend cell align=left, 
		align=left, 
		font=\tiny,
		draw=white!15!black,
		at={(0.02,1.0)},
		anchor=north west,},
]
\addplot [color=mycolor1, line width=\lineWidth]
 table [x index=0, y index=1, col sep=comma]{img/DataCSV/ParTB/Input1/Input1.csv};
 \addlegendentry{$ \calE_{\rr}(\prmtr,\inp,\fn)$ for \eqref{eqn:ParTB:inp1}}
 \addplot [color=mycolor2, dashed, line width=\lineWidth]
 table [x index=0, y index=1, col sep=comma]{img/DataCSV/ParTB/Input2/Input2.csv};
 \addlegendentry{$ \calE_{\rr}(\prmtr,\inp,\fn)$ for \eqref{eqn:ParTB:inp2}}
\addplot [color=mycolor3, dash dot, line width=\lineWidth]
 table [x index=0, y index=1, col sep=comma]{img/DataCSV/ParTB/Input3/Input3.csv};
 \addlegendentry{$ \calE_{\rr}(\prmtr,\inp,\fn)$ for \eqref{eqn:ParTB:inp3}}
\end{axis}

\end{tikzpicture}%

%% file: img/ParTB/Example_ParTB_2.tex
%
\begin{tikzpicture}

\begin{axis}[%
	width=\imageWidth,
	height=\imageHeight,
	scale only axis,
	grid=both,
	grid style={line width=.1pt, draw=gray!10},
	major grid style={line width=.2pt,draw=gray!50},
	axis lines*=left,
	axis line style={line width=\lineWidth},
    scaled x ticks = true,        
xmin=3,
xmax=30,
xlabel style={font=\color{white!15!black}},
xlabel={$J$}, 
ymode=log,
ymin=1e-13,
ymax=1e11,
yminorticks=true,
ylabel style={font=\color{white!15!black}},
	axis background/.style={fill=white},
	legend style={%
		legend cell align=left, 
		align=left, 
		font=\tiny,
		draw=white!15!black,
		at={(1.0,1.0)},
		anchor=north east,},
]

\addplot [color=mycolor1, line width=\lineWidth, forget plot]
 table [x index=0, y index=1, col sep=comma]{img/DataCSV/ParTB/Delta5.csv};

 \addplot [color=mycolor1, line width=\lineWidth]
 table [x index=0, y index=1, col sep=comma]{img/DataCSV/ParTB/TildeEpsilon.csv};
\addlegendentry{$\tilde \Delta_{56}(\mu_J)$ and $ \tilde \varepsilon_{56}$}


\addplot [color=mycolor2, line width=\lineWidth, dashed, forget plot]
 table [x index=0, y index=1, col sep=comma]{img/DataCSV/ParTB/Delta4.csv};

\addplot [color=mycolor2, dashed, line width=\lineWidth]
 table [x index=0, y index=2, col sep=comma]{img/DataCSV/ParTB/TildeEpsilon.csv};
\addlegendentry{$\tilde \Delta_{49}(\mu_J)$ and $ \tilde \varepsilon_{49}$}


\addplot [color=mycolor3, dashdotted, line width=\lineWidth, forget plot]
 table [x index=0, y index=1, col sep=comma]{img/DataCSV/ParTB/Delta3.csv};

\addplot [color=mycolor3, dashdotted, line width=\lineWidth]
 table [x index=0, y index=3, col sep=comma]{img/DataCSV/ParTB/TildeEpsilon.csv};
\addlegendentry{$\tilde \Delta_{42}(\mu_J)$ and $ \tilde \varepsilon_{42}$}


\addplot [color=mycolor4, dash dot dot, line width=1.1*\lineWidth, forget plot]
 table [x index=0, y index=1, col sep=comma]{img/DataCSV/ParTB/Delta2.csv};

\addplot [color=mycolor4, dash dot dot, line width=1.1*\lineWidth]
 table [x index=0, y index=4, col sep=comma]{img/DataCSV/ParTB/TildeEpsilon.csv};
\addlegendentry{$\tilde \Delta_{36}(\mu_J)$ and $ \tilde \varepsilon_{36}$}
 
\addplot [color=mycolor5, dotted, line width=1.2*\lineWidth, forget plot]
 table [x index=0, y index=1, col sep=comma]{img/DataCSV/ParTB/Delta1.csv};
 
\addplot [color=mycolor5, dotted, line width=1.2*\lineWidth]
 table [x index=0, y index=5, col sep=comma]{img/DataCSV/ParTB/TildeEpsilon.csv};
 
\addlegendentry{$\tilde \Delta_{29}(\mu_J)$ and $ \tilde \varepsilon_{29}$}
\end{axis}
\end{tikzpicture}%

%% file: img/SNM/Example_SNM_1.tex
%
\begin{tikzpicture}

\begin{axis}[%
width=\imageWidth,
	height=\imageHeight,
	scale only axis,
	grid=both,
	grid style={line width=.1pt, draw=gray!10},
	major grid style={line width=.2pt,draw=gray!50},
	axis lines*=left,
	axis line style={line width=\lineWidth},
    scaled x ticks = true,        
xmin=1,
xmax=2401,
xlabel style={font=\color{white!15!black}},
ymode=log,
ymin=3e-15,
ymax=7e-14,
yminorticks=true,
ylabel style={font=\color{white!15!black}},
	axis background/.style={fill=white},
	legend style={%
		legend cell align=left, 
		align=left, 
		font=\tiny,
		draw=white!15!black,
		at={(0.02,0.9)},
		anchor=north west,},
]
\addplot [color=mycolor1, dashed, line width=\lineWidth]
 table [x index=0, y index=1, col sep=comma]{img/DataCSV/SNM/Input.csv};
\end{axis}

\end{tikzpicture}%

%% file: img/SNM/Example_SNM_2.tex
%
\begin{tikzpicture}

\begin{axis}[%
	width=\imageWidth,
	height=\imageHeight,
	scale only axis,
	grid=both,
	grid style={line width=.1pt, draw=gray!10},
	major grid style={line width=.2pt,draw=gray!50},
	axis lines*=left,
	axis line style={line width=\lineWidth},
    scaled x ticks = true,        
xmin=1,
xmax=38,
xlabel style={font=\color{white!15!black}},
xlabel={$J$}, 
ymode=log,
ymin=5e-11,
ymax=1e12,
yminorticks=true,
ylabel style={font=\color{white!15!black}},
	axis background/.style={fill=white},
	legend style={%
		legend cell align=left, 
		align=left, 
		font=\tiny,
		draw=white!15!black,
		at={(1.0,1.0)},
		anchor=north east,},
]

\addplot [color=mycolor1, line width=\lineWidth, forget plot]
 table [x index=0, y index=1, col sep=comma]{img/DataCSV/SNM/Delta5.csv};

 \addplot [color=mycolor1, line width=\lineWidth]
  table[row sep=crcr]{%
1	    34.7113\\
2200	34.7113\\
};
\addlegendentry{$\tilde \Delta_{58}(\mu_J)$ and $ \tilde \varepsilon_{58}$}


\addplot [color=mycolor2, line width=\lineWidth, dashed, forget plot]
 table [x index=0, y index=1, col sep=comma]{img/DataCSV/ParTB/Delta4.csv};

\addplot [color=mycolor2, dashed, line width=\lineWidth]
  table[row sep=crcr]{%
1	    0.0020\\
2200	0.0020\\
};
\addlegendentry{$\tilde \Delta_{51}(\mu_J)$ and $ \tilde \varepsilon_{51}$}


\addplot [color=mycolor3, dashdotted, line width=\lineWidth, forget plot]
 table [x index=0, y index=1, col sep=comma]{img/DataCSV/SNM/Delta3.csv};

\addplot [color=mycolor3, dashdotted, line width=\lineWidth]
  table[row sep=crcr]{%
1	9.9209e-07\\
2200	9.9209e-07\\
};
\addlegendentry{$\tilde \Delta_{44}(\mu_J)$ and $ \tilde \varepsilon_{44}$}


\addplot [color=mycolor4, dash dot dot, line width=1.1*\lineWidth, forget plot]
 table [x index=0, y index=1, col sep=comma]{img/DataCSV/SNM/Delta2.csv};

\addplot [color=mycolor4, dash dot dot, line width=1.1*\lineWidth]
  table[row sep=crcr]{%
1	9.1946e-09\\
2200	9.1946e-09\\
};
\addlegendentry{$\tilde \Delta_{37}(\mu_J)$ and $ \tilde \varepsilon_{37}$}

\addplot [color=mycolor5, dotted, line width=1.2*\lineWidth, forget plot]
 table [x index=0, y index=1, col sep=comma]{img/DataCSV/SNM/Delta1.csv};
 
\addplot [color=mycolor5, dotted, line width=1.2*\lineWidth]
  table[row sep=crcr]{%
1	 5.6141e-10\\
2200 5.6141e-10\\
};
\addlegendentry{$\tilde \Delta_{30}(\mu_J)$ and $ \tilde \varepsilon_{30}$}
\end{axis}
\end{tikzpicture}%